\documentclass[11pt]{amsart}
\usepackage{amsmath, amsthm, amssymb} 
\usepackage{mathrsfs} 
\usepackage[bookmarks=false]{hyperref}

\newtheorem{theo}{Theorem}[section]

\newtheorem{letterthm}{Theorem}

\newtheorem{lettercor}[letterthm]{Corollary}

\newtheorem{cor}[theo]{Corollary}
\newtheorem{lem}[theo]{Lemma}
\newtheorem{prop}[theo]{Proposition}

\theoremstyle{definition}
\newtheorem{rem}[theo]{Remark}

\newtheorem{example}[theo]{Example}

\newtheorem{df}[theo]{Definition}

\newtheorem{step}{Step}
\newtheorem*{claim}{Claim}

\newcommand{\R}{\mathbf{R}}
\newcommand{\C}{\mathbf{C}}
\newcommand{\Z}{\mathbf{Z}}
\newcommand{\F}{\mathbf{F}}

\newcommand{\N}{\mathbf{N}}

\newcommand{\Ad}{\operatorname{Ad}}
\newcommand{\id}{\text{\rm id}}
\newcommand{\Tr}{\text{\rm Tr}}
\newcommand{\Leb}{\text{\rm Leb}}

\newcommand{\Aut}{\operatorname{Aut}}

\newcommand{\Inn}{\operatorname{Inn}}
\newcommand{\Out}{\operatorname{Out}}
\newcommand{\QN}{\mathcal{Q}\mathcal{N}}

\newcommand{\weak}{\mathord{\text{\rm weak}}}
\newcommand{\alg}{\mathord{\text{\rm alg}}}
\newcommand{\op}{\mathord{\text{\rm op}}}
\newcommand{\LL}{\mathord{\text{\rm L}}}
\newcommand{\ovt}{\mathbin{\overline{\otimes}}}

\newcommand{\cb}{\text{\rm cb}}

\newcommand{\dpr}{^{\prime\prime}}

\begin{document}

\title[Structure of ${\rm II_1}$ factors arising from free Bogoljubov actions]{Structure of ${\rm II_1}$ factors arising from free Bogoljubov actions of arbitrary groups}

\begin{abstract}
In this paper, we investigate several structural properties for crossed product ${\rm II_1}$ factors $M$ arising from free Bogoljubov actions associated with orthogonal representations $\pi : G \to \mathcal O(H_\R)$ of arbitrary countable discrete groups. Under fairly general assumptions on the orthogonal representation $\pi : G \to \mathcal O(H_\R)$, we show that $M$ does not have property Gamma of Murray and von Neumann. Then we show that any regular amenable subalgebra $A \subset M$ can be embedded into $\LL(G)$ inside $M$. Finally, when $G$ is assumed to be amenable, we locate precisely any possible amenable or Gamma extension of $\LL(G)$ inside $M$.
\end{abstract}

\author{Cyril Houdayer }

\address{CNRS-ENS Lyon \\
UMPA UMR 5669 \\
69364 Lyon cedex 7 \\
France}

\email{cyril.houdayer@ens-lyon.fr}

\thanks{Research supported by ANR grants AGORA and NEUMANN}



\maketitle

\section{Introduction and statement of the main results}

In classical probability theory, there is a well known construction that associates with any orthogonal representation $\pi : G \to \mathcal O(H_\R)$ of a countable discrete group $G$ a probability measure-preserving action $G \curvearrowright (X_\pi, \mu_\pi)$ on a standard probability space. This action is called the {\em Gaussian action} associated with the orthogonal representation $\pi$. By construction, the Koopman representation of the Gaussian action contains $\pi$ as a subrepresentation (see \cite[Appendix~D]{kechris}). For instance, when $\lambda_G : G \to \mathcal O(\ell^2_\R(G))$ is the left regular orthogonal representation, the Gaussian action $G \curvearrowright (X_{\lambda_G}, \mu_{\lambda_G})$ is nothing but the {\em Bernoulli shift} $G \curvearrowright ([0, 1]^G, \Leb^G)$.

In the framework of his free probability theory, Voiculescu \cite{voiculescu85} introduced in the mid 80s the analogue of the Gaussian construction in this setting: the {\em free Gaussian functor} (see also \cite[Chapter 2]{voiculescu92}). To any real Hilbert space $H_\R$, one associates a tracial von Neumann algebra, denoted by $\Gamma(H_\R)\dpr$, which is $\ast$-isomorphic to the free group factor $\LL(\F_{\dim H_\R})$ on $\dim H_\R$ generators. Within this framework, the free group factor $\Gamma(H_\R)\dpr$ is generated by semicircular elements $W(e)$, $e \in H_\R$, which enjoy the following {\em freeness} property: whenever $(e_i)_{i \geq 1}$ is an orthogonal family in $H_\R$, the family of noncommutative random variables $(W(e_i))_{i \geq 1}$ is $\ast$-free with respect to the canonical trace $\tau$ on $\Gamma(H_\R)\dpr$. As we will see in Section \ref{preliminaries}, the semicircular elements $W(e)$ can be alternatively regarded as {\em words} of length one. To any orthogonal representation $\pi : G \to \mathcal{O}(H_\R)$ of any countable discrete group $G$ corresponds a unique trace-preserving action $\sigma_\pi : G \curvearrowright \Gamma(H_\R)\dpr$ called the {\em free Bogoljubov action} associated with the orthogonal representation $\pi$. The action $\sigma_\pi$ satisfies the following relation: 
$$\sigma_\pi(g)(W(e)) = W(\pi(g) e), \forall e \in H_\R, \forall g \in G.$$
We refer to Section \ref{preliminaries} for more information on Voiculescu's free Gaussian functor. We will denote by $\Gamma(H_\R)\dpr \rtimes_\pi G$ the tracial crossed product von Neumann algebra corresponding to the free Bogoljubov action $\sigma_\pi : G \curvearrowright \Gamma(H_\R)\dpr$. For instance, when $\lambda_G : G \to \mathcal O(\ell^2_\R(G))$ is the left regular orthogonal representation, the free Bogoljubov action $\sigma_{\lambda_G} : G \curvearrowright \Gamma(\ell^2_\R(G))\dpr$ is nothing but the {\em free Bernoulli shift} $G \curvearrowright \ast_{g \in G} (\LL(\Z), \tau)$. In that case, the crossed product von Neumann algebra $\Gamma(\ell^2_\R(G))\dpr \rtimes_{\lambda_G} G$ is $\ast$-isomorphic to the free product von Neumann algebra $\LL(\Z) \ast \LL(G)$.

In this paper, we use Popa's deformation/rigidity theory \cite{{popa-icm}, {vaes-icm}, {ioana-ecm}} to investigate several structural properties for the crossed products ${\rm II_1}$ factors $\Gamma(H_\R)\dpr \rtimes_\pi G$ arising from free Bogoljubov actions of countable discrete groups. The first {\em rigidity} results for ${\rm II_1}$ factors arising from free Bernoulli shifts of property (T) groups were obtained by Popa in \cite{popa-bernoulli}, using his malleable deformation for the free group factors. In \cite{ipp}, Ioana, Peterson and Popa discovered a malleable deformation for amalgamated free product ${\rm II_1}$ factors which they used to obtain rigidity results for such factors and calculate their symmetry groups.

Popa \cite{popasup} discovered that in many previous arguments in deformation/rigidity theory, the property (T) condition could be removed and replaced by a {\em spectral gap} rigidity condition. This fundamental discovery lead to several {\em structural} results for ${\rm II_1}$ factors arising from free probability theory. For instance, Popa \cite{popa-solid} used his spectral gap rigidity principle to give another proof of Ozawa's result \cite{ozawa} showing that the free group factors are {\em solid}, that is, the relative commutant of any diffuse von Neumann subalgebra is amenable (we also refer to Peterson's work on $\LL^2$-derivations \cite{peterson-L2} and its applications). Subsequently, Chifan and the author \cite{chifan-houdayer} used the malleable deformation from \cite{ipp} together with Popa's principle \cite{popasup} to obtain structural properties, such as {\em primeness}, for a large class of amalgamated free products factors (see also \cite{houdayer-vaes}). Ozawa and Popa \cite{ozawa-popa} also used this spectral gap rigidity principle to prove that the free group factors are in fact {\em strongly solid}, that is, the normalizer of any diffuse amenable von Neumann subalgebra is amenable. This result strengthened both Voiculescu's result in \cite{voiculescu96} showing that the free group factors have no Cartan subalgebra and Ozawa's result in \cite{ozawa} showing that the free group factors are solid.

Recently, Shlyakhtenko and the author \cite{houdayer-shlyakhtenko} obtained several structural results, such as {\em absence of Cartan subalgebra}, for the ${\rm II_1}$ factors $\Gamma(H_\R)\dpr \rtimes_\pi G$ arising from free Bogoljubov actions of {\em amenable} groups. The amenability of $G$ was essential to ensure that the crossed product von Neumann algebra $\Gamma(H_\R)\dpr \rtimes_\pi G$ has the complete metric approximation property \cite{cowling-haagerup} in order to use Ozawa-Popa's results \cite{ozawa-popa}. For instance, it was proven in \cite[Theorem B]{houdayer-shlyakhtenko} that when the orthogonal representation $\pi : \Z \to \mathcal O(H_\R)$ is mixing, the crossed product ${\rm II_1}$ factor $\Gamma(H_\R)\dpr \rtimes_\pi \Z$ is strongly solid. This gave new examples of strongly solid ${\rm II_1}$ factors which are not $\ast$-isomorphic to interpolated free group factors (see also \cite{houdayer7}).

The aim of the paper is thus to generalize these previous results as well as to obtain new structural properties for the ${\rm II_1}$ factors $\Gamma(H_\R)\dpr \rtimes_\pi G$ arising from free  Bogoljubov actions associated with orthogonal representations $\pi : G \to \mathcal O(H_\R)$ of {\em arbitrary} countable discrete groups.

\subsection*{Property Gamma}

Our first result deals with property Gamma of Murray and von Neumann \cite{MvN}. Recall that a ${\rm II_1}$ factor $(M, \tau)$ has {\em property Gamma} if there exists a net of unitaries $u_i \in \mathcal U(M)$ such that $\tau(u_k) = 0$ for all $k$ and $\lim_k \|u_k y - y u_k\|_2 = 0$ for all $y \in M$. When $M$ has separable predual, Connes' result \cite[Corollary 3.8]{connes74} shows that $M$ does not have property Gamma if and only if the group of inner automorphisms $\Inn(M)$ is closed in the group of all automorphisms $\Aut(M)$. Observe that in that case, $\Out(M) = \Aut(M) / \Inn(M)$ is a Polish group \cite{connes74}.

Let $Q$ be a ${\rm II_1}$ factor with separable predual which does not have property Gamma. Denote by $\Pi : \Aut(Q) \to \Out(Q)$ the quotient homomorphism. In \cite[Theorem 1]{jones-full}, Jones proved that whenever $\sigma : G \to \Aut(Q)$ is a faithful action of a countable discrete group for which $\Pi(\sigma(G))$ is discrete in $\Out(Q)$, then the crossed product ${\rm II_1}$ factor $Q \rtimes_\sigma G$ does not have property Gamma.

Inspired by Jones' result, we find a sufficient condition on the orthogonal representation $\pi : G \to \mathcal O(H_\R)$ which ensures that the crossed product ${\rm II_1}$ factor $\Gamma(H_\R)\dpr \rtimes_\pi G$ does not have property Gamma.

\begin{letterthm}\label{thmA}
Let $G$ be any countable discrete group and $\pi : G \to \mathcal O(H_\R)$ any faithful orthogonal representation such that $\dim H_\R \geq 2$ and $\pi(G)$ is discrete in $\mathcal O(H_\R)$ with respect to the strong topology. Then $M = \Gamma(H_\R)\dpr \rtimes_\pi G$ is a ${\rm II_1}$ factor which does not have property Gamma.
\end{letterthm}

The proof of Theorem \ref{thmA} (see Section $\ref{gamma}$) does not actually use Jones' result but rather a combination of words techniques involving the generators $W(e)$, $e \in H_\R$, and methods from Popa's seminal article \cite{popa-amenable} on maximal amenable subalgebras in ${\rm II_1}$ factors. The key step (see Proposition \ref{commutant}) is to prove that when $\pi : G \to \mathcal O(H_\R)$ is an infinite dimensional orthogonal representation, then any central sequence of $\Gamma(H_\R)\dpr \rtimes_\pi G$ must asymptotically lie in $\LL(G)$.

When the group $G$ is {\em abelian} and $\pi : G \to \mathcal O(H_\R)$ is a faithful orthogonal representation such that $\dim H_\R \geq 2$, the sufficient condition in Theorem \ref{thmA} is also necessary, that is, $\Gamma(H_\R)\dpr \rtimes_\pi G$ is a ${\rm II_1}$ factor which does not have property Gamma if and only if $\pi(G)$ is discrete in $\mathcal O(H_\R)$ with respect to the strong topology (see Corollary \ref{equivalence-gamma}). Examples of orthogonal representations $\pi : G \to \mathcal O(H_\R)$ for which $\pi(G)$ is discrete in $\mathcal O(H_\R)$ include the ones which contain a mixing subrepresentation.

\subsection*{Regular amenable subalgebras}
Whenever $A \subset M$ is an inclusion of tracial von Neumann algebras, we denote by $\mathcal N_M(A) = \{u \in \mathcal U(M) : uAu^* = A\}$ the group of all the {\em normalizing unitaries} of $A$ inside $M$. Recall that $A \subset M$ is a {\em Cartan subalgebra} if $A \subset M$ is maximal abelian and $\mathcal N_M(A)\dpr = M$.

In their breakthrough article \cite{ozawa-popa}, Ozawa and Popa obtained a remarkable dichotomy result for {\em compact} actions of free groups. Let $\F_n \curvearrowright (X, \mu)$ be a compact probability measure-preserving (pmp) action of the free group onto $n$ generators ($n \geq 2$) on a standard probability space and put $M = \LL^\infty(X) \rtimes \F_n$. Ozawa and Popa \cite{ozawa-popa} proved that whenever $A \subset M$ is an amenable von Neumann subalgebra, then either $A \preceq_M \LL^\infty(X)$ or $\mathcal N_M(A)\dpr$ is amenable. We refer to Section \ref{preliminaries} for Popa's intertwining techniques and the symbol $\preceq_M$. In particular, any compact free ergodic pmp action $\F_n \curvearrowright (X, \mu)$ gives rise to a ${\rm II_1}$ factor $M = \LL^\infty(X) \rtimes \F_n$ with a unique Cartan decomposition, up to unitary conjugacy.

In a recent breakthrough paper \cite{popa-vaes-cartan1}, Popa and Vaes obtained a very general dichotomy result for {\em arbitrary} actions of free groups. Let $\F_n \curvearrowright (B, \tau)$ be an arbitrary trace-preserving action of $\F_n$ on a tracial von Neumann algebra $(B, \tau)$ and put $M = B \rtimes \F_n$. Popa and Vaes \cite[Theorem 1.6]{popa-vaes-cartan1} proved that whenever $A \subset M$ is a von Neumann subalgebra which is amenable relative to $B$ inside $M$, then either $A \preceq_M B$ or $\mathcal N_M(A)\dpr$ is amenable relative to $B$ inside $M$. We refer to Section \ref{preliminaries} for the notion of {\em relative amenability}. In particular, {\em any} free ergodic pmp action $\F_n \curvearrowright (X, \mu)$ gives rise to a ${\rm II_1}$ factor $M = \LL^\infty(X) \rtimes \F_n$ with a unique Cartan decomposition, up to unitary conjugacy. We refer to \cite{{ozawa-popa2}, {houdayer-shlyakhtenko}, {chifan-sinclair}, {chifan-sinclair-udrea}, {popa-vaes-cartan2}, {houdayer-vaes}} for further results in these directions.

Very recently, Ioana \cite{ioana-cartan} used a combination of Popa-Vaes' dichotomy result \cite{popa-vaes-cartan1} together with new word techniques to study Cartan subalgebras in amalgamated free product von Neumann algebras. One of the most general results Ioana obtained (see \cite[Theorem 1.6]{ioana-cartan}) is the following. Let $M = M_1 \ast_B M_2$ be an arbitrary tracial amalgamated free product. Let $A \subset M$ be a von Neumann subalgebra which is amenable relative to $B$ inside $M$ and $\omega \in \beta(\N) \setminus \N$ a free ultrafilter such that $\mathcal N_M(A)' \cap M^\omega = \C$, that is, $\mathcal N_M(A)\dpr$ has ``spectral gap" inside $M$. Then at least one of the following holds true:
\begin{itemize}
\item $A \preceq_M B$.
\item $\mathcal N_M(A)\dpr \preceq_M M_i$ for some $i \in \{1, 2\}$.
\item $\mathcal N_M(A)\dpr$ is amenable relative to $B$ inside $M$.
\end{itemize}
Very recenty, Vaes improved Ioana's dichotomy result (see \cite[Theorem A]{Va13}) by removing the spectral gap assumption $\mathcal N_M(A)' \cap M^\omega = \C$.

In this paper, we use Ioana's ideas and results from \cite{ioana-cartan} as well as Vaes' result \cite{Va13} to prove the following general dichotomy result for free Bogoljubov actions of {\em arbitrary} countable discrete groups $G$. This theorem should be compared to Popa-Vaes' result \cite[Theorem 1.6]{popa-vaes-cartan1}.

\begin{letterthm}\label{thmB}
Let $G$ be any countable discrete group and $\pi : G \to \mathcal O(H_\R)$ any orthogonal representation. Denote by $M = \Gamma(H_\R)\dpr \rtimes_\pi G$ the corresponding crossed product von Neumann algebra under the free Bogoljubov action $\sigma_\pi : G \curvearrowright \Gamma(H_\R)\dpr$. Let $p \in M$ be a nonzero projection and $A \subset pMp$ any von Neumann subalgebra that is amenable relative to $\LL(G)$ inside $M$. 

Then at least one of the following conclusions holds:
\begin{itemize}
\item $A \preceq_M \LL(G)$. 
\item $\mathcal N_{pMp}(A)\dpr$ is amenable relative to $\LL(G)$ inside $M$.
\end{itemize}
\end{letterthm}

Note that Theorem \ref{thmB} generalizes the main result of \cite{houdayer-shlyakhtenko}. Indeed, a similar result was proven in \cite[Theorem 3.5]{houdayer-shlyakhtenko} under the assumption that $G$ is {\em amenable}.

The proof of Theorem \ref{thmB} uses Ioana's original strategy \cite{ioana-cartan} and Vaes' result \cite{Va13} in the following way. To simplify, assume that $A \subset M$ is an amenable von Neumann subalgebra and put $P = \mathcal N_M(A)\dpr$. Assume that $P$ is not amenable relative to $\LL(G)$ inside $M$. Our aim is to show that $A \preceq_M \LL(G)$. We use Popa's malleable deformation $(\theta_t)$ on $\Gamma(H_\R \oplus H_\R)\dpr \rtimes_{\pi \oplus \pi} G$ arising from the {\em second quantization} of the one-parameter family of rotations on $H_\R \oplus H_\R$ that continuously map $H_\R \oplus 0$ onto $0 \oplus H_\R$. The key observation is that we can regard the crossed product von Neumann algebra $\widetilde M = \Gamma(H_\R \oplus H_\R)\dpr \rtimes_{\pi \oplus \pi} G$ as the amalgamated free product
$$\left( \Gamma(H_\R)\dpr \rtimes_\pi G \right) \ast_{\LL(G)} \left( \Gamma(H_\R)\dpr \rtimes_\pi G \right),$$
where we identify $M$ with the left copy of $\Gamma(H_\R)\dpr \rtimes_\pi G$ in the amalgamated free product. For $t > 0$ small enough, we now use Vaes' dichotomy result \cite[Theorem A]{Va13} for the inclusion $\theta_t(A) \subset \widetilde M$ and obtain that necessarily $\theta_t(A) \preceq_{\widetilde M} M$. In Section \ref{intertwining}, using word techniques involving the generators $W(e)$, $e \in H_\R$, we prove that this condition implies that $A \preceq_M \LL(G)$ (see Theorem \ref{intertwining2}).

The general dichotomy result obtained in Theorem \ref{thmB} together with Theorem \ref{thmA} allows us to obtain a new class of ${\rm II_1}$ factors with no Cartan subalgebra. 

\begin{lettercor}\label{corC}
Let $G$ be any countable discrete group and $\pi : G \to \mathcal O(H_\R)$ any faithful orthogonal representation such that $\dim H_\R \geq 2$. If $A \subset M$ is a regular amenable subalgebra then $A \preceq_M \LL(G)$.

Moreover, the following statements hold true:
\begin{enumerate}
\item If $\pi$ contains a direct sum of at least two finite dimensional subrepresentations, then $M$ has no Cartan subalgebra.
\item  If $\pi$ contains a mixing subrepresentation, then $M$ has no diffuse amenable regular von Neumann subalgebra.
\end{enumerate}
\end{lettercor}

Observe that in case when the orthogonal representation $\pi : G \to \mathcal O(H_\R)$ is {\em reducible}, the first part of Corollary \ref{corC} can be directly deduced from Vaes' result (see \cite[Theorem A]{Va13}). Indeed, if $\pi = \pi_1 \oplus \pi_2$ and $H_\R = H_\R^{(1)} \oplus H_\R^{(2)}$, then the crossed product ${\rm II_1}$ factor $\Gamma(H_\R)\dpr \rtimes_\pi G$ can be regarded as the amalgamated free product 
$$\left( \Gamma(H_\R^{(1)})\dpr \rtimes_{\pi_1} G \right) \ast_{\LL(G)} \left( \Gamma(H_\R^{(2)})\dpr \rtimes_{\pi_2} G \right)$$
and so Vaes' result \cite{Va13} can be applied. However, when $\pi : G \to \mathcal O(H_\R)$ is {\em irreducible}, the ${\rm II_1}$ factor $\Gamma(H_\R)\dpr \rtimes_\pi G$ no longer splits as an amalgamated free product over $\LL(G)$ and so in that case, Corollary \ref{corC} cannot be deduced from Vaes' result.

\subsection*{Maximal amenable and maximal Gamma extensions}

In his seminal article \cite{popa-amenable}, Popa proved that the generator masa in a free group factor is maximal amenable. In fact, Popa showed \cite[Lemma 2.1]{popa-amenable} that the generator masa in a free group factor satisfies the {\em asymptotic orthogonality property} (see Section $\ref{AOP}$ for further details). He then used this property to deduce that the generator masa is maximal amenable inside the free group factor (see \cite[Corollary 3.3]{popa-amenable}).

In the recent paper \cite{houdayer13}, we gave many new examples of maximal amenable masas in ${\rm II_1}$ factors by proving that whenever $G$ is an abelian group and $\pi : G \to \mathcal O(H_\R)$ is a mixing orthogonal representation, then $\LL(G)$ is maximal amenable inside $\Gamma(H_\R)\dpr \rtimes_\pi G$. This was done by showing that the inclusion $\LL(G) \subset \Gamma(H_\R)\dpr \rtimes_\pi G$ satisfies the asymptotic orthogonality property (see \cite[Theorem 3.2]{houdayer13}).

Very recently, Jesse Peterson asked us whether the maximal amenability of $\LL(G)$ inside the ${\rm II_1}$ factor $\Gamma(H_\R)\dpr \rtimes_\pi G$ could hold true under the more general assumption that $\pi : G \to \mathcal O(H_\R)$ is {\em weakly mixing}. We give a positive answer to his question and furthermore we prove the following theorem which generalizes the main result of \cite{houdayer13} and gives a new class of maximal amenable subalgebras in ${\rm II_1}$ factors. We will say that an orthogonal representation $\pi : G \to \mathcal O(H_\R)$ is {\em compact} if $\pi$ is a direct sum of finite dimensional orthogonal representations.

\begin{letterthm}\label{thmD}
Let $G$ be any amenable countable discrete group and $\pi : G \to \mathcal O(H_\R)$ any faithful orthogonal representation. Denote by $K_\R \subset H_\R$ the unique closed $\pi(G)$-invariant subspace such that $\pi_K = \pi | K_\R$ is weakly mixing and $\pi_{H \ominus K} = \pi | H_\R \ominus K_\R$ is compact. Put $M = \Gamma(H_\R)\dpr \rtimes_\pi G$ and $N = \Gamma(H_\R \ominus K_\R)\dpr \rtimes_{\pi_{H \ominus K}}G$.

Then for any  intermediate amenable von Neumann subalgebra $\LL(G) \subset P \subset M$, we have $P \subset N$.

In particular, if $\pi$ is weakly mixing, then $\LL(G)$ is maximal amenable inside $M$.
\end{letterthm}

As we will see in Section \ref{maximal-extensions}, Theorem \ref{thmD} will be deduced from a very general result regarding the {\em relative asymptotic orthogonality property} of the inclusion $N \subset M$ (see Theorem \ref{relative-AOP}).

Observe that the ${\rm II_1}$ factor $\Gamma(H_\R)\dpr \rtimes_\pi G$ may have property Gamma when $\pi$ is weakly mixing. This phenomenon cannot happen when $\pi$ is mixing by Theorem A. More generally, our last result below shows that when the group $G$ is amenable and the orthogonal representation $\pi : G \to \mathcal O(H_\R)$ contains a mixing subrepresentation, one can locate precisely not only the amenable extensions of $\LL(G)$ inside $\Gamma(H_\R)\dpr \rtimes_\pi G$ but also the Gamma extensions of $\LL(G)$ inside $\Gamma(H_\R)\dpr \rtimes_\pi G$, that is, the intermediate von Neumann subalgebras $\LL(G) \subset P \subset \Gamma(H_\R)\dpr \rtimes_\pi G$ which have property Gamma.

\begin{letterthm}\label{thmE}
Let $G$ be any amenable countable discrete group and $\pi : G \to \mathcal O(H_\R)$ any faithful orthogonal representation. Let $K_\R$ be a nonzero closed $\pi(G)$-invariant subspace such that $\pi | K_\R$ is mixing. Put $\pi_{H \ominus K} = \pi | H_\R \ominus K_\R$, $M = \Gamma(H_\R)\dpr \rtimes_\pi G$ and $N = \Gamma(H_\R \ominus K_\R)\dpr \rtimes_{\pi_{H \ominus K}}G$.

Then for any  intermediate von Neumann subalgebra $\LL(G) \subset P \subset M$ which has property Gamma, we have $P \subset N$.

In particular, if $N$ has property Gamma, then $N$ is the unique maximal Gamma extension of $\LL(G)$ inside $M$.
\end{letterthm}

\subsection*{Acknowledgments}

The present work was initiated when the author was staying at the Banff International Research Station for the ``Set Theory and Functional Analysis" workshop held in June 2012. He thanks the organizers for their invitation. 

The author is very grateful to Jesse Peterson for asking him whether \cite[Theorem 3.2]{houdayer13} could hold true under a weak mixing assumption. He is also very grateful to Stefaan Vaes for useful discussions which motivated the further generalizations stated above as Theorems \ref{thmD} and \ref{thmE}.

The author thanks Adrian Ioana for explaining \cite{ioana-cartan} to him and for useful discussions as well as R\'emi Boutonnet and Stefaan Vaes for their valuable comments. Finally, he thanks the referee for carefully reading the paper and providing useful remarks.

\subsection*{Notations}

All the groups $G$ that we consider in this paper are always assumed to be countable and discrete and the real Hilbert spaces $H_\R$ are always assumed to be separable. A {\em tracial} von Neumann algebra $(M, \tau)$ is a von Neumann algebra $M$ endowed with a faithful normal tracial state $\tau$. The uniform norm will be denoted by $\|x\|_\infty$ for all $x \in M$ while the $\LL^2$-norm associated with $\tau$ will be denoted by $\|x\|_2 = \tau(x^*x)^{1/2}$ for all $x \in M$. The unit ball of $M$ with respect to the uniform norm will be denoted by $(M)_1$.

\tableofcontents

\section{Preliminaries}\label{preliminaries}

\subsection{An elementary fact on $\varepsilon$-orthogonality}\label{elementary}

\begin{df}
Let $\mathcal H$ be a complex Hilbert space and $\varepsilon \geq 0$. We say that two (not necessarily closed) subspaces $\mathcal K, \mathcal L \subset \mathcal H$ are $\varepsilon$-{\em orthogonal} and we denote by $\mathcal K \perp_\varepsilon \mathcal L$ if
$$|\langle \xi, \eta \rangle_{\mathcal H} | \leq \varepsilon \, \|\xi\|_{\mathcal H} \, \|\eta\|_{\mathcal H}, \; \forall \xi \in \mathcal K, \forall \eta \in \mathcal L.$$
\end{df}

Define the function
$$\delta : \left[0, \frac12 \right) \to \R_+ :  t \mapsto \frac{2 t}{\sqrt{1 - t - \sqrt{2} \, t \sqrt{1 - t}}}.$$ 
We will be using the following elementary fact regarding $\varepsilon$-orthogonality whose proof  can be found in \cite[Proposition 2.3]{houdayer13}.

\begin{prop}[\cite{houdayer13}]\label{projections}
Let $k \geq 1$. Let $0 \leq \varepsilon < 1$ such that $\delta^{\circ (k - 1)}(\varepsilon) < 1/2$. For all $1 \leq i \leq 2^k$, let $p_i \in \mathbf B(\mathcal H)$ be projections such that $p_i \mathcal H \perp_\varepsilon p_j \mathcal H$ for all $i, j \in \{1, \dots , 2^k\}$ such that $i \neq j$. Write $P_k = \bigvee_{i = 1}^{2^k} p_i$. Then for all $\xi \in \mathcal H$, we have
$$\sum_{i = 1}^{2^k} \|p_i \xi\|_{\mathcal H}^2 \leq \prod_{j = 0}^{k - 1} \left( 1 + \delta^{\circ j}(\varepsilon) \right)^2 \|P_k \xi\|_{\mathcal H}^2.$$
\end{prop}

\subsection{Popa's intertwining techniques}\label{intertwining-techniques}

Let $Q \subset (M, \tau)$ be an inclusion of tracial von Neumann algebras. Jones' {\em basic construction} $\langle M, e_Q\rangle$ is the von Neumann subalgebra of $\mathbf B(\LL^2(M))$ generated by $M$ and the orthogonal projection $e_Q : \LL^2(M) \to \LL^2(Q)$. Recall that if we denote by $\rho : Q^{\op} \to \mathbf B(\LL^2(M))$ the right $Q$-action on $\LL^2(M)$, we have $\langle M, e_Q\rangle = \mathbf B(\LL^2(M)) \cap \rho(Q^{\op})'$. The basic construction $\langle M, e_Q\rangle$ is endowed with a canonical semifinite faithful normal trace $\Tr$ which satisfies
$$\Tr(x e_Q y) = \tau(xy), \forall x, y \in M.$$

In \cite{popa-malleable1, Po01}, Popa discovered the following powerful method to unitarily conjugate subalgebras of a tracial von Neumann algebra. Let $(M, \tau)$ be a tracial von Neumann algebra and $P\subset 1_P M 1_P$, $Q \subset 1_Q M 1_Q$ von Neumann subalgebras. By \cite[Corollary 2.3]{popa-malleable1} and \cite[Theorem A.1]{Po01} (see also \cite[Proposition C.1]{vaes-bourbaki}), the following conditions are equivalent:

\begin{itemize}
\item There exist $n \geq 1$, a projection $q \in \mathbf M_n(Q)$, a nonzero partial isometry $v \in \mathbf M_{1, n}(1_P M)q$ and a unital normal $\ast$-homomorphism $\varphi : P \to q\mathbf M_n(Q)q$  such that $a v = v \varphi(a)$ for all $a \in P$.

\item There exist projections $p \in P$ and $q \in Q$, a nonzero partial isometry $v \in pMq$ and a unital normal $\ast$-homomorphism $\varphi : pPp \to qQq$ such that $a v = v \varphi(a)$ for all $a \in P$.

\item There is no net of unitaries $(w_k)$ in $P$ such that
$$\lim_k \|E_Q(x^* w_k y)\|_2 = 0, \forall x, y \in 1_P M 1_Q.$$
\end{itemize}

If one of the previous equivalent conditions is satisfied, we say that $A$ {\it embeds into} $B$ {\it inside} $M$ and write $A \preceq_M B$.

Following \cite{{jones-index},{pimsner-popa}}, we say that an inclusion of tracial von Neumann algebras $Q \subset (M, \tau)$ has {\em finite index} if $\LL^2(M, \tau)$ has finite dimension as a right $Q$-module. 

\begin{rem}\label{remark-intertwining}
Let $(M, \tau)$ be a tracial von Neumann algebra and $P \subset 1_P M 1_P$ and $Q \subset 1_Q M 1_Q$ von Neumann subalgebras. If $A \subset P$ is a von Neumann subalgebra with finite index and if $A \preceq_M Q$, then $P \preceq_M Q$ (see \cite[Lemma 3.9]{vaes-bimodules}).
\end{rem}

\subsection{Hilbert bimodules}

Let $(M, \tau)$ and $(N, \tau)$ be any tracial von Neumann algebras. Recall that an $M$-$N$-{\em bimodule} $\mathcal H$ is a Hilbert space endowed with two commuting normal $\ast$-representations $\pi : M \to \mathbf B(\mathcal H)$ and $\rho : N^{\op} \to \mathbf B(\mathcal H)$. We then define $\pi_{\mathcal H} : M \otimes_{\alg} N^{\op} \to \mathbf B(\mathcal H)$ by $\pi_{\mathcal H}(x \otimes y^{\op}) = \pi(x) \rho(y^{\op})$ for all $x \in M$ and all $y \in N$. We will simply write $x \xi y = \pi_{\mathcal H}(x \otimes y^{\op}) \xi$ for all $x \in M$, all $y \in N$ and all $\xi \in \mathcal H$. The $N$-$N$-bimodule $\LL^2(N)$ with left and right action given by $x \xi y = x J y^*J \xi$ is the {\em trivial} $N$-$N$-bimodule while the $N$-$N$-bimodule $\LL^2(N) \otimes \LL^2(N)$ with left and right action given by $x (\xi \otimes \eta) y = x \xi \otimes J y^* J \eta$ is the {\em coarse} $N$-$N$-bimodule.

Let $\mathcal H$ and $\mathcal K$ be $M$-$N$-bimodules. Following \cite[Appendix V.B]{Co94}, we say that $\mathcal K$ is {\em weakly contained} in $\mathcal H$ and write $\mathcal K \subset_{\weak} \mathcal H$ if $\|\pi_{\mathcal K}(T) \|_\infty \leq \|\pi_{\mathcal H}(T)\|_\infty$ for all $T \in M \otimes_{\alg} N^{\op}$.

For any tracial von Neumann algebras $(B, \tau)$, $(M, \tau)$, $(N, \tau)$, any $M$-$B$-bimodule $\mathcal H$ and any $B$-$N$-bimodule $\mathcal K$, there is a well defined $M$-$N$-bimodule $\mathcal H \otimes_B \mathcal K$ called the {\em Connes' fusion tensor product} of $\mathcal H$ and $\mathcal K$ over $B$. We refer to \cite[Appendix V.B]{Co94} and \cite[Section 1]{AD93} for more details regarding this construction.

\subsection{Relative amenability}

Whenever $P \subset \mathcal N$ is an inclusion of von Neumann algebras, a positive functional $\varphi$ on $\mathcal N$ is $P$-{\em central} if $\varphi(x T) = \varphi(T x)$ for all $T \in \mathcal N$ and all $x \in P$.

Recall from \cite{connes76} that a tracial von Neumann algebra $(P, \tau)$ is {\em amenable} if there exists a $P$-central state $\varphi$ on $\mathbf B(\LL^2(P))$ such that $\varphi | P = \tau | P$. By Connes' celebrated result \cite{connes76}, a tracial von Neumann algebra $P$ with separable predual is amenable if and only if it is hyperfinite. 

\begin{df}[\cite{ozawa-popa}]
Let $(M, \tau)$ be a tracial von Neumann algebra, $p \in M$ a nonzero projection and $P \subset pMp$, $Q \subset M$ von Neumann subalgebras. We say that $P$ is {\em amenable relative to} $Q$ {\em inside} $M$ if there exists a $P$-central positive functional $\varphi$ on $p \langle M, e_Q \rangle p$ such that $\varphi | pMp =  \tau | pMp$.
\end{df}

By \cite[Theorem 2.1]{ozawa-popa}, $P$ is amenable relative to $Q$ inside $M$ if and only if there exists a net of vectors $\xi_k \in \LL^2(p\langle M, e_Q\rangle p, \Tr)$ such that $\lim_k \|y \xi_k  - \xi_k y\|_{2, \Tr} = 0$ for all $y \in P$ and $\lim_k \langle x \xi_k, \xi_k\rangle_{\Tr} = \tau(x)$ for all $x \in pMp$. This is equivalent to the fact the $pMp$-$P$-bimodule $p \LL^2(M) p$ is weakly contained in the $pMp$-$P$-bimodule $p \LL^2(M) \otimes_Q \overline{\LL^2(M)}p$.

\begin{rem}\label{remark}
We will be using the following facts. Let $(M, \tau)$ be a tracial von Neumann algebra and $P \subset pMp$ a von Neumann subalgebra.
\begin{enumerate}
\item If $P$ is amenable relative to $Q$ inside $M$ and if $A \subset eMe$ is a von Neumann subalgebra which satisfies $A \preceq_M P$, then there exists a nonzero projection $f \in A' \cap eMe$ such that $Af$ is amenable relative to $Q$ inside $M$ (see \cite[Section 2.4]{ipv}).

\item If $P$ is amenable relative to $Q$ inside $M$, and $e \in P$, $f \in P' \cap pMp$ are projections, then $e P ef$ is amenable relative to $Q$ inside $M$.

\item If $P p_1$ is amenable relative to $Q$ inside $M$ for some nonzero projection $p_1 \in P' \cap pMp$, then $P p_2$ is amenable relative to $Q$ inside $M$ with $p_2 \in \mathcal Z(P' \cap pMp)$ the central support of $p_1$ inside $P' \cap pMp$ (see \cite[Remark 2.2]{ioana-cartan}).
\end{enumerate}
\end{rem}

\subsection{Voiculescu's free Gaussian functor}

Let $H_\R$ be a separable real Hilbert space. Let $H = H_\R \otimes_\R \C = H_\R \oplus {\rm i} H_\R$ be the corresponding complexified Hilbert space. The canonical complex conjugation on $H$ will be simply denoted by $\overline{e + {\rm i} f} = e - {\rm i} f$ for all $e, f \in H_\R$. The \emph{full Fock space} of $H$ is defined by
\begin{equation*}
\mathcal{F}(H) =\C\Omega \oplus \bigoplus_{n \geq 1} H^{\otimes n}.
\end{equation*}
The unit vector $\Omega$ is called the \emph{vacuum vector}. For all $e \in H$, we define the \emph{left creation operator}
$$
\ell(e) : \mathcal{F}(H) \to \mathcal{F}(H) : \left\{ 
{\begin{array}{l} \ell(e)\Omega = e \\ 
\ell(e)(e_1 \otimes \cdots \otimes e_n) = e \otimes e_1 \otimes \cdots \otimes e_n.
\end{array}} \right.
$$
We have $\ell(e)^* \ell(f) = \langle e, f\rangle$ for all $e, h \in H$. In particular, $\ell(e)$ is an isometry for all unit vector $e \in H$.

For all $e \in H_\R$, put $W(e) = \ell(e) + \ell(e)^*$. Voiculescu's result \cite[Lemma 2.6.3]{voiculescu92} shows that the distribution of the selfadjoint operator $W(e)$ with respect to the vacuum vector state $\langle \cdot \Omega, \Omega\rangle$ is the semicircular law supported on the interval $[-2 \| e \|, 2 \| e \|]$. Moreover, \cite[Lemma 2.6.6]{voiculescu92} shows that for every subset $\Xi \subset H_\R$ of pairwise orthogonal vectors, the family $(W(e))_{e \in \Xi}$ is freely independent with respect to $\langle \cdot \Omega, \Omega\rangle$.

We denote by $\Gamma(H_\R)$ the C$^*$-algebra generated by $\{W(e) : e \in H_\R\}$ and by $\Gamma(H_\R)\dpr$ the von Neumann algebra generated by $\Gamma(H_\R)$. The vector state $\tau = \langle \cdot \Omega, \Omega\rangle$ is a faithful normal trace on $\Gamma(H_\R)\dpr$ and $\Gamma(H_\R)\dpr$ is $\ast$-isomorphic to the free group factor on $\dim H_\R$ generators, that is, $\Gamma(H_\R)\dpr \cong \LL(\F_{\dim H_\R})$.

Since the vacuum vector $\Omega$ is separating and cyclic for $\Gamma(H_\R)\dpr$, any
$x \in \Gamma(H_\R)\dpr$ is uniquely determined by $\xi = x \Omega \in \mathcal{F}(H)$. Thus we will write $x = W(\xi)$. Note that for $e \in H_\R$, we recover the semicircular random variables $W(e) = \ell(e) + \ell(e)^*$ generating $\Gamma(H_\R)\dpr$. More generally we have $W(e) = \ell(e) + \ell(\overline e)^*$ for all $e \in H$. Given any vectors $e_i \in H$, it is easy to check that $e_1\otimes \cdots \otimes e_n$ lies in $\Gamma(H_\R)\dpr \Omega$. The corresponding {\em words} $W(e_1 \otimes \cdots \otimes e_n) \in \Gamma(H_\R)\dpr$ enjoy useful properties that are summarized in the following result.

\begin{prop}[\cite{houdayer13}]\label{wick}
Let $e_i , f_j \in H$, for $i, j \geq 1$. The following are true:
\begin{enumerate}
\item We have the {\em Wick formula}:
$$W(e_1 \otimes \cdots \otimes e_n) = \sum_{k = 0}^n \ell(e_1) \cdots \ell(e_k) \ell(\overline e_{k + 1})^* \cdots \ell(\overline e_n)^*.$$
\item  We have that $W(e_1 \otimes \cdots \otimes e_r) W(f_1 \otimes \cdots \otimes f_s)$ is equal to
$$
W(e_1 \otimes \cdots \otimes e_r \otimes f_1 \otimes \cdots \otimes f_s) + 
\langle \overline e_r, f_1\rangle W(e_1 \otimes \cdots \otimes e_{r - 1}) W(f_2 \otimes \cdots \otimes f_s)
$$
\item We have $W(e_1 \otimes \cdots \otimes e_n)^* = W(\overline e_n \otimes \cdots \otimes \overline e_1)$.
\item The linear span of $\{1, W(e_1 \otimes \cdots \otimes e_n) : n \geq 1, e_i \in H \}$ forms a unital weakly dense $\ast$-subalgebra of $\Gamma(H_\R)\dpr$.
\end{enumerate}
\end{prop}

\begin{proof}
The proof of $(1)$ is borrowed from \cite[Lemma 3.2]{houdayer-ricard}. We prove the formula by induction on $n$. For $n \in \{0, 1\}$, we have $W(\Omega)=1$ and we already observed that $W(e_i)=\ell(e_i) + \ell(\overline e_i)^*$.

Next, for $e_{0}\in H$, we have 
\begin{align*}
W(e_{0})W(e_1\otimes \cdots \otimes e_n)\Omega & = W(e_{0})(e_1\otimes \cdots \otimes e_n)
\\ 
& = (\ell(e_{0})+\ell(\overline e_{0})^*)  e_1\otimes \cdots \otimes e_n \\ 
& = e_{0}\otimes e_1\otimes \cdots \otimes e_n + \langle \overline e_{0},e_{1}\rangle \, e_2\otimes \cdots \otimes e_n.
\end{align*}
So, we obtain 
\begin{align*}
W(e_0\otimes \cdots \otimes e_n) & = W(e_{0})W(e_1\otimes \cdots \otimes e_n) - \langle \overline e_{0},e_{1}\rangle W(e_2\otimes \cdots \otimes e_n) \\
& = \ell(\overline e_{0})^*W(e_1\otimes \cdots \otimes e_n) - \langle \overline e_{0},e_{1}\rangle W(e_2\otimes \cdots \otimes e_n) \\
& \ \ \ + \ell(e_{0}) W(e_1\otimes \cdots \otimes e_n).
\end{align*}
Using the assumption for $n$ and $n-1$ and the relation $\ell(\overline e_0)^*\ell(e_1) = \langle \overline e_0, e_1\rangle$, we obtain
$$\ell(\overline e_{0})^*W(e_1\otimes \cdots \otimes e_n)=\langle \overline
e_{0},e_{1}\rangle W(e_2\otimes \cdots \otimes e_n) + \ell(\overline
e_{0})^*\ell(\overline e_{1})^* \cdots \ell(\overline e_{n})^*.$$ 
Since $\ell(e_{0})W(e_1\otimes \cdots \otimes e_n)$ gives the last $n + 1$ terms in
the Wick formula at order $n+1$ and $\ell(\overline e_{0})^*\ell(\overline e_{1})^* \cdots \ell(\overline e_{n})^*$ gives the first term, we are done.

$(2)$ By the Wick formula, we have that $W(e_1 \otimes \cdots \otimes e_r) W(f_1 \otimes \cdots \otimes f_s)$ is equal to 
$$\sum_{0 \leq j \leq r, 
0 \leq k \leq s} \ell(e_1) \cdots \ell(e_j) \ell(\overline e_{j + 1})^* \cdots \ell(\overline e_r)^* \ell(f_1) \cdots \ell(f_k) \ell(\overline f_{k + 1})^* \cdots \ell(\overline f_s)^*.
$$
Recall that we have $\ell(\overline e_r)^*\ell(f_1) = \langle \overline e_r, f_1 \rangle$. Therefore the above sum simply equals
\begin{align*}
&  \left( \sum_{0 \leq j \leq r - 1} \ell(e_1) \cdots \ell(e_j) \ell(\overline e_{j + 1})^*\cdots \ell(\overline e_r)^* \ell(\overline f_1)^* \cdots \ell(\overline f_s)^* \right. \\
& \left. + \sum_{0 \leq k \leq s} \ell(e_1) \cdots \ell(e_r) \ell(f_1) \cdots \ell(f_k) \ell(\overline f_{k + 1})^* \cdots \ell(\overline f_s)^* \right) \\
&+ \langle \overline e_r, f_1\rangle \sum_{0 \leq j \leq r - 1,  1 \leq k \leq s} 
\ell(e_1) \cdots \ell(e_j) \ell(\overline e_{j + 1})^* \cdots \ell(\overline e_{r - 1})^* \ell(f_2) \cdots \ell(f_k) \ell(\overline f_{k + 1})^* \cdots \ell(\overline f_s)^*.
\end{align*}
Therefore $W(e_1 \otimes \cdots \otimes e_r) W(f_1 \otimes \cdots \otimes f_s)$ is equal to 
$$
W(e_1 \otimes \cdots \otimes e_r \otimes f_1 \otimes \cdots \otimes f_s) + 
\langle \overline e_r, f_1\rangle W(e_1 \otimes \cdots \otimes e_{r - 1}) W(f_2 \otimes \cdots \otimes f_s).
$$

$(3)$ This is a straightforward consequence of $(1)$. 

$(4)$ This is a straightforward consequence of $(3)$ using an induction procedure.
\end{proof}

Let $G$ be any countable discrete group and $\pi : G \to \mathcal{O}(H_\R)$ any orthogonal representation.  We shall still denote by $\pi : G \to \mathcal{U}(H)$ the corresponding unitary representation on the complexified Hilbert space $H = H_\R \otimes_\R \C$. The {\it free Bogoljubov action} $\sigma_\pi : G \curvearrowright (\Gamma(H_\R)\dpr, \tau)$ associated with the orthogonal representation $\pi$ is defined by
$$
\sigma_\pi(g) = \Ad(\rho(g)), \forall g \in G,
$$
where $\rho(g) = \id_{\C \Omega} \oplus \bigoplus_{n \geq 1} \pi(g)^{\otimes n} \in \mathcal{U}(\mathcal{F}(H))$. We will also sometimes more generally write $\mathcal F(U) = \id_{\C \Omega} \oplus \bigoplus_{n \geq 1} U^{\otimes n}$ for all $U \in \mathcal U(H)$. Observe that we have
$$\sigma_\pi(g)(W(e_1 \otimes \cdots \otimes e_n)) = W(\pi(g) e_1 \otimes \cdots \otimes \pi(g)e_n)$$
for all $n \geq 1$ and all $e_i \in H$.

\begin{example}
If $\lambda_G : G \to \mathcal O(\ell_\R^2(G))$ is the left regular orthogonal representation of $G$, then the action $\sigma_{\lambda_G} : G \curvearrowright \Gamma(\ell^2_\R(G))\dpr$ is the free Bernoulli shift and in that case we have 
$$\left( \LL(G) \subset \Gamma(\ell^2_\R(G))\dpr \rtimes_{\lambda_G} G \right)  \cong \left( \LL(G) \subset \LL(\Z) \ast \LL(G) \right).$$
\end{example}

Recall that an orthogonal representation $\pi : G \to \mathcal O(H_\R)$ is {\em mixing} if $\lim_{g \to \infty} \langle \pi(g)\xi, \eta \rangle = 0$ for all $\xi, \eta \in H_\R$.

\begin{prop}[\cite{houdayer13}]\label{mixing-action}
Let $G$ be any countable discrete group and $\pi : G \to \mathcal{O}(H_\R)$ any orthogonal representation. The following are equivalent:
\begin{enumerate}
\item The representation $\pi : G \to \mathcal{O}(H_\R)$ is mixing.
\item The $\tau$-preserving action $\sigma_\pi : G \curvearrowright \Gamma(H_\R)\dpr$ is mixing, that is, 
\begin{equation*}
\lim_{g \to \infty} \tau(\sigma_\pi(g)(x) y) = 0, \forall x, y \in \Gamma(H_\R)\dpr \ominus \C.
\end{equation*}
\end{enumerate}
\end{prop}

Finally, recall from \cite[Theorem 5.1]{houdayer-shlyakhtenko} that whenever the orthogonal representation $\pi : G \to \mathcal O(H_\R)$ is faithful, the associated free Bogoljubov action $\sigma_\pi : G \curvearrowright \Gamma(H_\R)\dpr$ is {\em properly outer}, that is, $\sigma_\pi(g) \notin \Inn(\Gamma(H_\R)\dpr)$ for all $g \in G \setminus \{1\}$. In that case, we have
$$\Gamma(H_\R)' \cap (\Gamma(H_\R)\dpr \rtimes_\pi G) = \Gamma(H_\R)' \cap \Gamma(H_\R)\dpr = \C$$
and so $\Gamma(H_\R)\dpr \rtimes_\pi G$ is a ${\rm II_1}$ factor.

\subsection{The malleable deformation on $\Gamma(H_\R)\dpr \rtimes_\pi G$}\label{deformation}

Let $G$ be any countable discrete group and $\pi : G \to \mathcal{O}(H_\R)$ any orthogonal representation. Put
\begin{itemize}
\item $M = \Gamma(H_\R)\dpr \rtimes_\pi G$.
\item $\widetilde{M} = \Gamma(H_\R \oplus H_\R)\dpr \rtimes_{\pi \oplus \pi} G$.
\end{itemize}
We can regard $\widetilde{M}$ as the amalgamated free product
\begin{equation*}
\widetilde{M} = \left( \Gamma(H_\R)\dpr \rtimes_\pi G \right) \ast_{\LL(G)} \left( \Gamma(H_\R)\dpr \rtimes_\pi G \right),
\end{equation*}
where we identify $M$ with the left copy of $\Gamma(H_\R)\dpr \rtimes_\pi G$ inside the amalgamated free product. Consider the following orthogonal transformations on $H_\R \oplus H_\R$:
$$
V  =  
\begin{pmatrix}
1 & 0 \\
0 & -1
\end{pmatrix}  \; \mbox{ and } \;
U_t  =  
\begin{pmatrix}
\cos(\frac{\pi}{2} t) & -\sin(\frac{\pi}{2} t) \\
\sin(\frac{\pi}{2} t) & \cos(\frac{\pi}{2} t)
\end{pmatrix}, \forall t \in \R.
$$

Define the associated deformation $(\theta_t, \beta)$ on $\Gamma(H_\R \oplus H_\R)''$ by 
\begin{equation*}
\theta_t = \Ad(\mathcal F (U_t)) \; \mbox{ and } \; \beta = \Ad(\mathcal{F}(V)). 
\end{equation*}
Since $U_t$ and $V$ commute with $\pi \oplus \pi$, it follows that $\alpha_t$ and $\beta$ commute with the diagonal action $\sigma_\pi \ast \sigma_\pi$. We can then extend the deformation $(\theta_t, \beta)$ to $\widetilde{M}$ after defining $\theta_t | \LL(G) = \beta | \LL(G) = \id$. Moreover it is easy to check that the deformation $(\theta_t, \beta)$ is {\it malleable} in the sense of Popa:
\begin{enumerate}
\item $\lim_{t \to 0} \|x - \theta_t(x)\|_2 = 0$, $\forall x \in \widetilde{M}$.
\item $\beta^2 = \id$ and $\theta_t \beta = \beta \theta_{-t}$, $\forall t \in \R$.
\end{enumerate}
Since $\theta_t , \beta \in \Aut(\widetilde M)$ are trace-preserving, we will also denote by $\theta_t, \beta \in \mathcal U(\LL^2(\widetilde M))$ the corresponding Koopman unitary operators.

For all $0 < \rho \leq 1$, denote by ${\rm m}_{\rho} : M \to M$ the trace-preserving unital completely positive multiplier which satisfies
$${\rm m}_\rho (W(e_1 \otimes \cdots \otimes e_n) u_g) = \rho^n W(e_1 \otimes \cdots \otimes e_n) u_g.$$
With $\rho_t = \cos(\frac{\pi}{2} t)$, a straightforward calculation yields $E_M \circ \theta_t = {\rm m}_{\rho_t}$ for all $t \in \R$. In this respect, $(\theta_t)_{t \in \R}$ is a {\em dilation} of the one-parameter family $({\rm m}_{\rho_t})_{t \in \R}$ of unital completely positive maps on $M$. 

Denote by $\mathcal H_n = H^{\otimes n}$ the closed linear subspace of $\mathcal F(H)$ of all the words $e_1 \otimes \cdots \otimes e_n$ of length $n \geq 1$. By convention, denote $\mathcal H_0 = \C \Omega$. We have 
$$\LL^2(M) = \bigoplus_{n \in \N} (\mathcal H_n \otimes \ell^2(G)).$$

\begin{prop}\label{calculation}
Let $t \in [-1, 1]$, $x \in M$ and write $x = \sum_{n \in \N} \xi_n$ where $\xi_n \in \mathcal H_n \otimes \ell^2(G)$. The following hold:
\begin{enumerate}
\item $\tau(\theta_t(x) x^*) = \sum_{n \in \N} \rho_t^{n} \|\xi_n\|_2^2$.
\item $\frac12 \|x - \theta_t(x)\|_2^2 \leq \|(E_M\circ\theta_t)(x) - \theta_t(x)\|_2^2$.
\end{enumerate}
\end{prop}

\begin{proof}
For $(1)$, observe that $\tau(\theta_t(x) x^*) = \tau(E_M(\theta_t(x)) x^*) = \sum_{n \in \N} \rho_t^n \|\xi_n\|_2^2$. 

For $(2)$, observe that 
$$\|(E_M \circ \theta_t)(x) - \theta_t(x)\|_2^2 = \|x\|_2^2 - \|(E_M \circ \theta_t)(x)\|_2^2 = \sum_{n \in \N} (1 - \rho_t^{2n}) \|\xi_n\|_2^2$$ 
and
$$\|x - \theta_t(x)\|_2^2 = 2 (\|x\|_2^2 - \Re \tau(\theta_t(x) x^*)) = 2 \sum_{n \in \N} (1 - \rho_t^n) \|\xi_n\|_2^2.$$
Since $0 \leq \rho_t \leq 1$ for all $t \in [-1, 1]$, we obtain $\frac12 \|x - \theta_t(x)\|_2^2 \leq \|(E_M \circ\theta_t)(x) - \theta_t(x)\|_2^2$.
\end{proof}

We say that a von Neumann subalgebra $P \subset \Gamma(H_\R)\dpr \rtimes_\pi G$ is $(\theta_t)$-{\em rigid} if $(\theta_t)$ converges to $\id$ in $\|\cdot\|_2$ uniformly on the unit ball $(P)_1$. The next theorem shows that any $(\theta_t)$-rigid von Neumann subalgebra $P \subset \Gamma(H_\R)\dpr \rtimes_\pi G$ can be embedded into $\LL(G)$ inside $\Gamma(H_\R)\dpr \rtimes_\pi G$.

\begin{theo}\label{intertwining1}
Let $G$ be any countable discrete group and $\pi : G \to \mathcal{O}(H_\R)$ any orthogonal representation. Put  $M = \Gamma(H_\R)\dpr \rtimes_\pi G$. Let $p \in M$ be a non-zero projection. Let  $P \subset pMp$ be a von Neumann subalgebra and assume that there exist $c > 0$ and $t \in (-1, 0) \cup (0, 1)$ such that
$$\tau(\theta_t(u) u^*) \geq c, \forall u \in \mathcal U(P).$$
Then $P \preceq_M \LL(G)$.
\end{theo}

\begin{proof}
Let $c > 0$ and $t \in (-1, 0) \cup (0, 1)$ such that $\tau(\theta_t(u) u^*) \geq c$ for all $u \in \mathcal U(P)$. By Proposition \ref{calculation}, we have that $t \mapsto \tau(\theta_t(x) x^*)$ is an even function which is decreasing on $[0, 1]$ for all $x \in M$. We can find $n \in \N$ large enough so that $2^{-n} \leq | t |$. Thus $\tau(\theta_{2^{-n}}(u) u^*) \geq \tau(\theta_t(u)u^*) \geq c$ for all $u \in \mathcal U(P)$. Now the rest of the proof is entirely identical to the one of \cite[Theorem 4.3]{houdayer-ricard} (see also \cite[Theorem 5.2]{houdayer3}) and leads to $P \preceq_M \LL(G)$.
\end{proof}

\section{Intertwining subalgebras in ${\rm II_1}$ factors $\Gamma(H_\R)\dpr \rtimes_\pi G$}\label{intertwining}

We keep the same notation as in Section \ref{deformation}. The aim of this section is to prove the following intertwining theorem for subalgebras of $\Gamma(H_\R)\dpr \rtimes_\pi G$ which is inspired by \cite[Theorem 3.2]{ioana-cartan}.

\begin{theo}\label{intertwining2}
Let $G$ be any countable discrete group and $\pi : G \to \mathcal{O}(H_\R)$ any orthogonal representation. Put  $M = \Gamma(H_\R)\dpr \rtimes_\pi G$. Let $p \in M$ be a nonzero projection and $P \subset pMp$ a von Neumann subalgebra. Let $t \in (-1, 0) \cup (0, 1)$ such that $\theta_t(P) \preceq_{\widetilde M} M$. Then $P \preceq_M \LL(G)$.
\end{theo}

The proof of Theorem \ref{intertwining2} relies on the following convergence result.

\begin{theo}\label{convergence}
Let $G$ be any countable discrete group and $\pi : G \to \mathcal{O}(H_\R)$ any orthogonal representation. Put  $M = \Gamma(H_\R)\dpr \rtimes_\pi G$. Let $t \in (-1, 0) \cup (0, 1)$ and a net $x_k \in (M)_1$  such that $\lim_{k} \tau(\theta_t(x_k) x_k^*) = 0$. Then
$$\lim_{k} \|E_M(a \theta_t(x_k) b)\|_2 = 0, \forall a, b \in \widetilde M.$$
\end{theo}

\begin{proof}[Proof of Theorem $\ref{intertwining2}$ using Theorem $\ref{convergence}$]
Assume $P \npreceq_M \LL(G)$. Let $t \in (-1, 0) \cup (0, 1)$. By Theorem $\ref{intertwining1}$, there exists a net of unitaries $u_k \in \mathcal U(P)$ such that $\lim_k \tau(\theta_t(u_k) u_k^*) = 0$. By Theorem $\ref{convergence}$, we get $\lim_k \|E_M(a \theta_t(u_k) b)\|_2 = 0$ for all $a, b \in \widetilde M$, whence $\theta_t(P) \npreceq_{\widetilde M} M$.
\end{proof}

The proof of Theorem \ref{convergence} relies on the following technical result. As usual $H = H_\R \otimes_\R \C$ denotes the complexified space of $H_\R$. Put $\rho(g) = \id_{\C \Omega} \oplus \bigoplus_{n \geq 1} \pi(g)^{\otimes n}$ for all $g \in G$. We denote by $\mathcal F(H)$ the full Fock space of $H$.

Put $M = \Gamma(H_\R)\dpr \rtimes_\pi G$. We will identify $\LL^2(M)$ with $\mathcal F(H) \otimes \ell^2(G)$ and denote by $\mathcal J : \mathcal F(H) \otimes \ell^2(G) \to \mathcal F(H) \otimes \ell^2(G)$ the conjugation defined by $\mathcal J \Omega = \Omega$ and
$$\mathcal J (e_1 \otimes \cdots \otimes e_n \otimes \delta_g) = \pi(g)^* \overline e_n \otimes \cdots \otimes \pi(g)^* \overline e_1 \otimes \delta_{g^{-1}}$$ 
for all $n \geq 1$, all $e_i \in H$ and all $g \in G$. 

Likewise, put $\widetilde M = \Gamma(H_\R \oplus H_\R)\dpr \rtimes_\pi G$. We will identify $\LL^2(\widetilde M)$ with $\mathcal F(H \oplus H) \otimes \ell^2(G)$ and denote by $\widetilde{\mathcal J} : \mathcal F(H \oplus H) \otimes \ell^2(G) \to \mathcal F(H \oplus H) \otimes \ell^2(G)$ the conjugation defined by $\widetilde{\mathcal J} \Omega = \Omega$ and 
$$ \widetilde{\mathcal J} (e_1 \otimes \cdots \otimes e_n \otimes \delta_g) = \pi(g)^* \overline e_n \otimes \cdots \otimes \pi(g)^* \overline e_1 \otimes \delta_{g^{-1}}$$ 
for all $n \geq 1$, all $e_i \in H \oplus  H$ and all $g \in G$.

We view $M \subset \widetilde M$ by identifying $M$ with $\Gamma(H_\R \oplus 0)\dpr \rtimes_{\pi \oplus \pi} G$ inside $\widetilde M$. We will denote by $E_M : \widetilde M \to M$ the trace-preserving conditional expectation as well as the orthogonal projection $\LL^2(\widetilde M) \to \LL^2(M)$.

Denote by $\mathcal H_n = H^{\otimes n}$ the closed linear span in $\mathcal F(H)$ of all the words $e_1 \otimes \cdots \otimes e_n$ of length $n \geq 1$. By convention, denote $\mathcal H_0 = \C \Omega$.

\begin{lem}\label{technical-lemma}
Let $t \in (-1, 0) \cup (0, 1)$. Assume that
\begin{itemize}
\item $a = 1$ or $a =  W(\xi_1 \otimes \cdots \otimes \xi_r)$ is a word of length $r \geq 1$ in $\Gamma(H_\R \oplus H_\R)\dpr$ with letters $\xi_i$ in $H \oplus 0$ or $0 \oplus H$ and such that $\xi_1 \in 0 \oplus H$.
\item $b = 1$ or $b = W(\eta_1 \otimes \cdots \otimes \eta_s)$ is a word of length $s \geq 1$ in $\Gamma(H_\R \oplus H_\R)\dpr$ with letters $\eta_j$ in $H \oplus 0$ or $0 \oplus H$ and such that $\eta_s \in 0 \oplus H$. 
\end{itemize}

Put $\kappa_n = \sup \left \{ \| E_M(a \widetilde{\mathcal J} b^* \widetilde{\mathcal J} \theta_t(\zeta)) \|_2 : \zeta \in \mathcal H_n \otimes \ell^2(G), \|\zeta\|_2 \leq 1 \right \}$. Then $\lim_{n \to \infty} \kappa_n = 0$.
\end{lem}

\begin{proof}
We may and will assume that $\|\xi_i\| = \|\eta_j\| = 1$ for all $1 \leq i \leq r$ and $1 \leq j \leq s$. Fix $\mathcal B = \{ e_i : i \geq 1 \}$ an orthonormal basis for $H$. Then
$$\mathcal B_n = \{e_{i_1} \otimes \cdots \otimes e_{i_n} : i_1, \dots, i_n \geq 1\}$$
forms an orthonormal basis for $\mathcal H_n$. Whenever $\zeta \in \mathcal H_n \otimes \ell^2(G)$, write $\zeta = \sum_{w \in \mathcal B_n, g \in G} \zeta_{w, g} \, w \otimes \delta_g$ with $\zeta_{w, g} \in \C$ such that $\sum_{w \in \mathcal B_n, g \in G} | \zeta_{w, g} |^2 = \|\zeta\|_2^2$.

We assume that $n \geq r + s + 1$ and $r, s \neq 0$.  Fix now $g \in G$ and $w \in \mathcal B_n$ that we write $w = e_{i_1} \otimes \cdots \otimes e_{i_n}$ for $i_1, \dots, i_n \geq 1$. We have $\theta_t(w \otimes \delta_g) = U_t e_{i_1} \otimes \cdots \otimes U_t e_{i_n} \otimes \delta_g$. We have
$$a \widetilde{\mathcal J} b^* \widetilde{\mathcal J} \theta_t(w \otimes \delta_g) = W(\xi_1 \otimes \cdots \otimes  \xi_r) W(U_t e_{i_1} \otimes \cdots \otimes U_t e_{i_n})W(\pi(g) \eta_1 \otimes \cdots \otimes \pi(g) \eta_s)\Omega \otimes \delta_g.$$

Applying repeatedly Proposition \ref{wick}, we have that $a \widetilde{\mathcal J} b^* \widetilde{\mathcal J} \theta_t(w \otimes \delta_g)$ is equal to
\begin{align*}
& W(\xi_1 \otimes \cdots \otimes  \xi_r \otimes U_t e_{i_1} \otimes \cdots \otimes U_t e_{i_n} \otimes \pi(g) \eta_1 \otimes \cdots \otimes \pi(g) \eta_s)\Omega \otimes \delta_g \\
&+ \langle \overline \xi_r, U_t e_{i_1}\rangle W(\xi_1 \otimes \cdots \otimes \xi_{r - 1}) W(U_t e_{i_2} \otimes \cdots \otimes U_t e_{i_n} \otimes \pi(g) \eta_1 \otimes \cdots \otimes \pi(g) \eta_s))\Omega \otimes \delta_g \\
&+ \langle U_t \overline e_{i_n}, \pi(g)\eta_1 \rangle W(\xi_1 \otimes \cdots \otimes  \xi_r \otimes U_t e_{i_1} \otimes \cdots \otimes U_t e_{i_{n - 1}})W(\pi(g) \eta_2 \otimes \cdots \otimes \pi(g) \eta_s) \Omega \otimes \delta_g \\
&+ \langle \overline \xi_r, U_t e_{i_1}\rangle \langle U_t \overline e_{i_n}, \pi(g)\eta_1 \rangle W(\xi_1 \otimes \cdots \otimes \xi_{r - 1}) W(U_t e_{i_2} \otimes \cdots \otimes U_t e_{i_{n - 1}}) W(\pi(g) \eta_2 \otimes \cdots \otimes \pi(g) \eta_s) \Omega \otimes \delta_g.
\end{align*}

Applying repeatedly Proposition \ref{wick} and using the facts that $\eta_s \in 0 \oplus H$ and $n \geq r + 1$, we have
$$E_M \left(W(\xi_1 \otimes \cdots \otimes \xi_{r - 1}) W(U_t e_{i_2} \otimes \cdots \otimes U_t e_{i_n} \otimes \pi(g) \eta_1 \otimes \cdots \otimes \pi(g) \eta_s))\Omega \otimes \delta_g \right) = 0$$

Likewise, applying repeatedly Proposition \ref{wick} and using the facts that $\xi_1 \in 0 \oplus H$ and $n \geq s + 1$, we have
$$E_M \left( W(\xi_1 \otimes \cdots \otimes  \xi_r \otimes U_t e_{i_1} \otimes \cdots \otimes U_t e_{i_{n - 1}})W(\pi(g) \eta_2 \otimes \cdots \otimes \pi(g) \eta_s) \Omega \otimes \delta_g \right) = 0.$$

Moreover, since $\xi_1, \eta_s \in 0 \oplus H$, we have
$$E_M \left( W(\xi_1 \otimes \cdots \otimes  \xi_r \otimes U_t e_{i_1} \otimes \cdots \otimes U_t e_{i_n} \otimes \pi(g) \eta_1 \otimes \cdots \otimes \pi(g) \eta_s)\Omega \otimes \delta_g \right) = 0.$$

Repeating this procedure by induction and using again repeatedly Proposition $\ref{wick}$, we finally obtain that 
$$E_M(a \widetilde{\mathcal J} b^* \widetilde{\mathcal J} \theta_t(w \otimes \delta_g)) = \rho_t^{n - r - s}\prod_{k = 1}^r \langle \overline \xi_{r - k + 1}, U_t e_{i_k} \rangle \prod_{l = 1}^s \langle U_t \overline e_{i_{n - l + 1}}, \pi(g) \eta_l \rangle \, e_{i_{r + 1}} \otimes \cdots \otimes e_{i_{n - s}} \otimes \delta_g,$$
with $\rho_t = \cos(\frac{\pi}{2} t)$. Observe that the above formula is still valid when $a = 1$, that is $r = 0$, or $b = 1$, that is, $s = 0$. This shows in particular that 
\begin{equation}\label{orthogonal-vectors}
E_M \left(a \widetilde{\mathcal J} b^* \widetilde{\mathcal J} \theta_t \left( \mathcal H_n \otimes \ell^2(G) \right) \right) \subset \mathcal H_{n - r - s} \otimes \ell^2(G).
\end{equation}

Whenever $w \in \mathcal B_n$, denote by $\mathcal T(w) \in \mathcal B_{n - r - s}$ the word obtained by removing the first $r$ letters and the last $s$ letters from $w$. In other words, if $w = e_{i_1} \otimes \cdots \otimes e_{i_n} \in \mathcal B_n$, we have $\mathcal T(w) = e_{i_{r + 1}} \otimes \cdots \otimes e_{i_{n - s}} \in \mathcal B_{n - r - s}$. What we have shown before can be rewritten as
\begin{align*}
E_M(a \widetilde{\mathcal J} b^* \widetilde{\mathcal J} \theta_t(w \otimes \delta_g)) & = \rho_t^{n - r - s} \prod_{k = 1}^r \langle U_t^* \overline \xi_{r - k + 1}, e_{i_k} \rangle \prod_{l = 1}^s \langle U_t^* \pi(g) \overline \eta_l, e_{i_{n - l + 1}} \rangle \, \mathcal T(w) \otimes \delta_g \\
& =  \rho_t^{n - r - s}  \left \langle U_t^* \overline \xi_r \otimes \cdots \otimes U_t^* \overline \xi_1 \otimes U_t^* \pi(g) \overline \eta_s \otimes \cdots \otimes U_t^* \pi(g) \overline \eta_1, u \right \rangle \mathcal T(w) \otimes \delta_g
\end{align*}
with $u = e_{i_1} \otimes \cdots \otimes e_{i_r} \otimes e_{i_{n - s + 1}} \otimes \cdots \otimes e_{i_n} \in \mathcal B_{r + s}$.

Recall that 
$$\zeta = \sum_{w \in \mathcal B_n, g \in G} \zeta_{w, g} \, w \otimes \delta_g = \sum_{g \in G} \sum_{v \in \mathcal B_{n - r - s}} \left( \sum_{w \in \mathcal B_n, \mathcal T(w) = v} \zeta_{w, g} \, w \otimes \delta_g \right).$$
Observe that for every $v \in \mathcal B_{n - r - s}$, there is a canonical one-to-one correspondence between $\mathcal B_{r + s}$ and $\{w \in \mathcal B_n : \mathcal T(w) = v\}$ via the map $\iota_v : \mathcal B_{r + s} \to \{w \in \mathcal B_n : \mathcal T(w) = v\}$ defined by
$$\iota_v \left (e_{i_1} \otimes \cdots \otimes e_{i_r} \otimes e_{i_{n - s + 1}} \otimes \cdots \otimes e_{i_n} \right) = e_{i_1} \otimes \cdots \otimes e_{i_r} \otimes v \otimes e_{i_{n - s + 1}} \otimes \cdots \otimes e_{i_n}.$$

We have that $\sum_{w \in \mathcal B_n, \mathcal T(w) = v} \zeta_{w, g} \, E_M(a \widetilde{\mathcal J} b^* \widetilde{\mathcal J} \theta_t(w \otimes \delta_g))$ is equal to 
$$\rho_t^{n - r - s} \left( \sum_{u \in \mathcal B_{r + s}} \zeta_{\iota_v(u), g} \left \langle U_t^* \overline \xi_r \otimes \cdots \otimes U_t^* \overline \xi_1 \otimes U_t^* \pi(g) \overline \eta_s \otimes \cdots \otimes U_t^* \pi(g) \overline \eta_1, u \right \rangle \right) v \otimes \delta_g.$$
Since $(u)_{u \in \mathcal B_{r + s}}$ is an orthonormal family and since $\|\xi_i\| = \|\eta_j\| = 1$ for all $1 \leq i \leq r$ and all $1 \leq j \leq s$, the Cauchy-Schwarz inequality yields
$$\left \| \sum_{w \in \mathcal B_n, \mathcal T(w) = v} \zeta_{w, g} \, E_M(a \widetilde{\mathcal J} b^* \widetilde{\mathcal J} \theta_t(w \otimes \delta_g)) \right\|_2^2 \leq \rho_t^{2(n - r - s)} \sum_{u \in \mathcal B_{r + s}} |\zeta_{\iota_v(u), g}|^2.$$

Altogether, we finally obtain
\begin{align*}
\| E_M(a \widetilde{\mathcal J} b^* \widetilde{\mathcal J} \theta_t(\zeta)) \|_2^2 &= \sum_{g \in G} \sum_{v \in \mathcal B_{n - r - s}} \left \| \sum_{w \in \mathcal B_n, \mathcal T(w) = v} \zeta_{w, g} \, E_M(a \widetilde{\mathcal J} b^* \widetilde{\mathcal J} \theta_t(w \otimes \delta_g)) \right\|_2^2 \\
&\leq \rho_t^{2(n - r - s)} \sum_{g \in G} \sum_{v \in \mathcal B_{n - r - s}} \left( \sum_{w \in \mathcal B_n, \mathcal T(w) = v} |\zeta_{w, g}|^2 \right) \\
&= \rho_t^{2(n - r - s)} \|\zeta\|_2^2.
\end{align*}

Recall that $\kappa_n = \sup \left \{ \| E_M(a \widetilde{\mathcal J} b^* \widetilde{\mathcal J} \theta_t(\zeta)) \|_2 : \zeta \in \mathcal H_n \otimes \ell^2(G), \|\zeta\|_2 \leq 1 \right \}$. We get $\kappa_n \leq \rho_t^{n - r - s}$. Since $t \in (-1, 0) \cup (0, 1)$, we have $0 \leq \rho_t < 1$, whence $\lim_{n \to \infty} \rho_t^{n - r - s} = 0$ and so $\lim_{n \to \infty} \kappa_n = 0$. This finishes the proof of Lemma \ref{technical-lemma}.
\end{proof}

\begin{proof}[Proof of Theorem $\ref{convergence}$]
Observe that using a combination of Proposition $\ref{wick}$ and Kaplansky's density theorem, it suffices to show that $\lim_{k} \|E_M(a \theta_t(x_k) b)\|_2 = 0$ for:
\begin{itemize}
\item $a = 1$ or $a = W(\xi_1 \otimes \cdots \otimes \xi_r)$ a word of length $r \geq 1$ in $\Gamma(H_\R \oplus H_\R)\dpr$ with letters $\xi_i$ in $H \oplus 0$ or $0 \oplus H$ and such that $\xi_1 \in 0 \oplus H$.

\item $b = 1$ or $b = W(\eta_1 \otimes \cdots \otimes \eta_s)$ a word of length $s \geq 1$ in $\Gamma(H_\R \oplus H_\R)\dpr$ with letters $\eta_j$ in $H \oplus 0$ or $0 \oplus H$ and such that $\eta_s \in 0 \oplus H$.
\end{itemize}

Write $x_k = \sum_{n \in \N} \xi_{k, n}$ with $\xi_{k, n} \in \mathcal H_n \otimes \ell^2(G)$. We then have
$\tau(\theta_t(x_k) x_k^*) = \sum_{n \in \N} \rho_t^n \|\xi_{k, n}\|_2^2$. Since $\lim_{k} \tau(\theta_t(x_k) x_k^*) = 0$ and $0 < \rho_t < 1$, we get $\lim_{k} \|\xi_{k, n}\|_2 = 0$ for all $n \in \N$. Put
$$\kappa_n = \sup \left \{ \| E_M(a \widetilde{\mathcal J} b^* \widetilde{\mathcal J} \theta_t(\zeta)) \|_2 : \zeta \in \mathcal H_n \otimes \ell^2(G), \|\zeta\|_2 \leq 1 \right \}.$$

Recall that by $(\ref{orthogonal-vectors})$ in the proof of Lemma $\ref{technical-lemma}$, we have 
$$E_M \left(a \widetilde{\mathcal J} b^* \widetilde{\mathcal J} \theta_t \left( \mathcal H_n \otimes \ell^2(G) \right) \right) \subset \mathcal H_{n - r - s} \otimes \ell^2(G)$$ 
for all $n \geq r + s + 1$. This implies that for every $k$, the vectors $\left( E_M (a \widetilde{\mathcal J} b^* \widetilde{\mathcal J} \theta_t ( \xi_{k,n} ) ) \right)_{n \geq {r + s + 1}}$ are pairwise orthogonal in $\LL^2(M)$. 

For all $k$, we get 
\begin{align*}
\|E_M(a \theta_t(x_k) b)\|_2^2 &= \left\| \sum_{n \leq r + s} E_M(a \widetilde{\mathcal J} b^* \widetilde{\mathcal J} \theta_t(\xi_{k, n})) + \sum_{n \geq r + s + 1} E_M(a \widetilde{\mathcal J} b^* \widetilde{\mathcal J} \theta_t(\xi_{k, n}) ) \right\|_2^2 \\
&\leq 2 \left\| \sum_{n \leq r + s} E_M(a \widetilde{\mathcal J} b^* \widetilde{\mathcal J} \theta_t(\xi_{k, n}) )) \right\|_2^2 + 2 \left\| \sum_{n \geq r + s + 1} E_M(a \widetilde{\mathcal J} b^* \widetilde{\mathcal J} \theta_t(\xi_{k, n}) ) \right\|_2^2 \\
&= 2 \left\| \sum_{n \leq r + s} E_M(a \widetilde{\mathcal J} b^* \widetilde{\mathcal J} \theta_t(\xi_{k, n}) )) \right\|_2^2 + 2 \sum_{n \geq r + s + 1} \left \| E_M(a \widetilde{\mathcal J} b^* \widetilde{\mathcal J} \theta_t(\xi_{k, n}) ) \right\|_2^2 \\
&\leq  2 \left\| \sum_{n \leq r + s} E_M(a \widetilde{\mathcal J} b^* \widetilde{\mathcal J} \theta_t(\xi_{k, n}) )) \right\|_2^2 + 2 \sum_{n \geq r + s + 1} \kappa_n^2 \|\xi_{k, n}\|_2^2. 
\end{align*}

Let $\varepsilon > 0$. Since $\lim_{n \to \infty} \kappa_n = 0$ by Lemma \ref{technical-lemma}, there exists $n_0 \geq r + s + 1$ such that $\kappa_n \leq \varepsilon / 2$ for all $n \geq n_0$. Since moreover $\lim_{k} \|\xi_{k, n}\|_2 = 0$ for all $n \in \N$, there exists $k_0$ such that for all $k \geq k_0$, we have 
$$2 \left\| \sum_{n \leq r + s} E_M(a \widetilde{\mathcal J} b^* \widetilde{\mathcal J} \theta_t(\xi_{k, n}) )) \right\|_2^2 + 2 \sum_{r + s + 1 \leq n \leq n_0 - 1} \kappa_n^2 \|\xi_{k, n}\|_2^2 \leq \frac{\varepsilon^2}{2}.$$
For all $k \geq k_0$, we obtain
\begin{align*}
\|E_M(a \theta_t(x_k) b)\|_2^2 &\leq \frac{\varepsilon^2}{2} +  2 \sum_{n \geq n_0} \kappa_n^2 \|\xi_{k, n}\|_2^2 \\
&\leq \frac{\varepsilon^2}{2} + \frac{\varepsilon^2}{2} \sum_{n \geq n_0} \|\xi_{k, n}\|_2^2 \\
& \leq  \frac{\varepsilon^2}{2} + \frac{\varepsilon^2}{2} \|x_k\|_2^2 \leq \varepsilon^2.
\end{align*}
This shows that $\lim_{k} \|E_M(a \theta_t(x_k) b)\|_2 = 0$ and finishes the proof of Theorem \ref{convergence}.
\end{proof}

\section{(Weakly) mixing inclusions in ${\rm II_1}$ factors $\Gamma(H_\R)\dpr \rtimes_\pi G$}

Let $P \subset Q$ be an inclusion of von Neumann algebras. Following \cite[Section 1.4.2]{Po01}, the {\em quasi-normalizer of} $P$ {\em inside} $Q$, denoted by $\QN_Q(P)$, is the set of all $x \in Q$ for which there exist $y_1, \dots, y_k \in Q$ such that 
$$x P \subset \sum_{i = 1}^k P y_i \; \mbox{ and } \; P x \subset \sum_{i = 1}^k y_i P.$$
One checks that $\QN_Q(P)$ is a unital $\ast$-subalgebra of $Q$ such that $P \vee (P' \cap Q) \subset \QN_Q(P)$. We say that $P$ is {\em quasi-regular inside} $Q$ if $\QN_Q(P)\dpr = Q$. Moreover by \cite[Lemma 3.5]{popa-malleable1}, for all projections $p \in P$ and $q \in P' \cap Q$, we have $pq \QN_Q(P)\dpr pq = \QN_{pq Q pq}(p Pq p)\dpr$.

\subsection{Weakly mixing inclusions in $\Gamma(H_\R)\dpr \rtimes_\pi G$}

The following definition is due to Popa and Vaes (see \cite[Definition 6.13]{popa-vaes-advances}).

\begin{df}
Let $A \subset N \subset (M, \tau)$ be tracial von Neumann algebras. We say that the inclusion $N \subset M$ is {\em weakly mixing through} $A$ if there exists a net of unitaries $u_k \in \mathcal U(A)$ such that 
$$\lim_k \|E_N(x u_k y)\|_2 = 0, \forall x, y \in M\ominus N.$$ 
\end{df}

The following result will be useful in order to prove Theorem \ref{thmD}. Recall that an orthogonal representation $\pi : G \to \mathcal O(H_\R)$ is {\em compact} if $\pi$ is the direct sum of finite dimensional orthogonal representations.

\begin{prop}\label{weak-mixing-through}
Let $G$ be any countable discrete group and $\pi : G \to \mathcal O(H_\R)$ any orthogonal representation. Denote by $K_\R \subset H_\R$ the unique closed $\pi(G)$-invariant subspace such that $\pi_K = \pi | K_\R$ is weakly mixing and $\pi_{H \ominus K} = \pi | H_\R \ominus K_\R$ is compact. Put $M = \Gamma(H_\R)\dpr \rtimes_\pi G$ and $N = \Gamma(H_\R \ominus K_\R)\dpr \rtimes_{\pi_{H \ominus K}} G$.

Then the inclusion $N  \subset M$ is weakly mixing through $\LL(G)$.
\end{prop}

\begin{proof}
As usual, we denote by $H$ (resp.\ $K$) the complexified Hilbert space of $H_\R$ (resp.\ $K_\R$). If $K_\R = 0$, then $N = M$ and the inclusion $M \subset M$ is trivially weakly mixing through $\LL(G)$. Thus, we may and will assume that $K_\R \neq 0$.

Since $\pi_K$ is weakly mixing, there exists a sequence $g_n \in G$ such that $\lim_n \langle \pi(g_n) \xi, \eta\rangle = 0$ for all $\xi, \eta \in K$. Observe that by Kaplansky's density theorem, in order to show that the inclusion $N  \subset M$ is weakly mixing through $\LL(G)$, it suffices to show that $\lim_n \|E_N (x u_{g_n} y) \|_2 = 0$ for all $x, y \in M \ominus N$ words of the form $x = W(\xi_1 \otimes \cdots \otimes \xi_r)$ and $y = W(\eta_1 \otimes \cdots \otimes \eta_s)$ with $r, s \geq 1$, letters $\xi_i, \eta_j$ in $K$ or $H \ominus K$ and $\xi_1, \eta_s \in K$.

Applying repeatedly Proposition $\ref{wick}$ together with the fact that $\xi_1, \eta_s \in K$, we have 
\begin{align*}
E_N (x u_{g_n} y) &= E_N \left( W(\xi_1 \otimes \cdots \otimes \xi_r) u_{g_n} W(\eta_1 \otimes \cdots \otimes \eta_s) \right) \\
&= E_N \left( W(\xi_1 \otimes \cdots \otimes \xi_r) W(\pi(g_n)\eta_1 \otimes \cdots \otimes \pi(g_n)\eta_s) \right) u_{g_n} \\
&= \delta_{r = s} \prod_{i = 1}^r \langle \overline \xi_{r - i + 1}, \pi(g_n) \eta_i \rangle u_{g_n}.
\end{align*}
Since $\xi_1, \eta_s \in K$, we have $\lim_n \langle \overline \xi_1, \pi(g_n)\eta_s \rangle = 0$, whence $\lim_n \|E_N (x u_{g_n} y)\|_2 = 0$.
\end{proof}

\begin{cor}\label{weak-mixing-consequence}
Let $G$ be any countable discrete group and $\pi : G \to \mathcal O(H_\R)$ any orthogonal representation. Denote by $K_\R \subset H_\R$ the unique closed $\pi(G)$-invariant subspace such that $\pi_K = \pi | K_\R$ is weakly mixing and $\pi_{H \ominus K} = \pi | H_\R \ominus K_\R$ is compact. Put $M = \Gamma(H_\R)\dpr \rtimes_\pi G$ and $N = \Gamma(H_\R \ominus K_\R)\dpr \rtimes_{\pi_{H \ominus K}} G$.

Whenever $x \in M$ satisfies $\LL(G) x \subset \sum_{i = 1}^k y_i N$ for some finite subset $\{y_1, \dots, y_k\} \subset M$, then $x \in N$. In particular, $\QN_M(\LL(G))\dpr \subset N$.
\end{cor}

\begin{proof}
This is a straightforward consequence of Proposition \ref{weak-mixing-through} and \cite[Proposition 6.14]{popa-vaes-advances}.
\end{proof}

\subsection{Mixing inclusions in $\Gamma(H_\R)\dpr \rtimes_\pi G$}

This next definition is motivated by Popa's result \cite[Theorem 3.1]{popa-malleable1} (see also \cite[Definition 9.1]{ioana-cartan}).

\begin{df}
Let $B \subset (M, \tau)$ be tracial von Neumann algebras. We say that the inclusion $B \subset M$ is {\em mixing} if whenever $b_k \in (B)_1$ is a net such that $b_k \to 0$ weakly, we have
$$\lim_k \|E_B(x b_k y)\|_2 = 0, \forall x, y \in M \ominus B.$$
\end{df}

If $G \curvearrowright (B, \tau)$ is a trace-preserving mixing action of a countable discrete group on a tracial von Neumann algebra, then the inclusion $\LL(G) \subset B \rtimes G$ is mixing. For other examples, we refer to \cite[Section 9.3]{ioana-cartan} and the references therein. 

\begin{rem}\label{remark-mixing}
Let $B \subset (M, \tau)$ be a mixing inclusion of tracial von Neumann algebras.
\begin{enumerate}
\item For all $k \geq 1$ and all projections $p \in \mathbf M_k(B)$, the inclusion $p\mathbf M_k(B) p \subset p \mathbf M_k(M) p$ is mixing.
\item Let $A \subset B$ be any diffuse von Neumann subalgebra. Then the inclusion $B \subset M$ is weakly mixing through $A$.
\end{enumerate}
\end{rem}

The aim of this section is to prove the following result that will be needed in the proof of Theorem \ref{thmE}.

\begin{prop}\label{mixing-inclusion}
Let $G$ be any countable discrete group and $\pi : G \to \mathcal O(H_\R)$ any orthogonal representation. Let $K_\R \subset H_\R$ be a nonzero closed $\pi(G)$-invariant subspace such that $\pi | K_\R$ is mixing. Put $\pi_{H \ominus K} = \pi | H_\R \ominus K_\R$, $M = \Gamma(H_\R)\dpr \rtimes_\pi G$ and $N = \Gamma(H_\R \ominus K_\R)\dpr \rtimes_{\pi_{H \ominus K}}G$.

Then the inclusion $N \subset M$ is mixing. 
\end{prop}

Proposition \ref{mixing-inclusion} will be a consequence of the following more general result about  mixing inclusions in tracial amalgamated free product von Neumann algebras.

\begin{prop}\label{mixing-afp}
For all $i \in \{1, 2\}$, let $B \subset (M_i, \tau_i)$ be an inclusion of tracial von Neumann algebras. Assume that $\tau_1|_B = \tau_2 |_B$ and denote by $(M, \tau) = (M_1 \ast_B M_2, \tau_1 \ast_B \tau_2)$ the corresponding tracial amalgamated free product von Neumann algebra. If the inclusion $B \subset M_2$ is mixing, so is the inclusion $M_1 \subset M$.
\end{prop}

\begin{proof}
Assume that the inclusion $B \subset M_2$ is mixing. Let $b_k \in (M_1)_1$ be any net of elements such that $b_k \to 0$ weakly. In order to prove that the inclusion $M_1 \subset M$ is mixing, it suffices to show that for all the reduced words of the form $x = x_r \cdots x_1$ and $y = y_1 \cdots y_s$ with $r, s \geq 2$, $x_1 = 1$ or $x_1 \in M_1 \ominus B$, $x_2 \in M_{i_2} \ominus B, \dots, x_r \in M_{i_r} \ominus B$ with $2 = i_2 \neq \cdots \neq i_r$, and $y_1 = 1$ or $y_1 \in M_1 \ominus B$, $y_2 \in M_{j_2} \ominus B, \dots, y_s \in M_{j_s} \ominus B$ with $2 = j_2 \neq \cdots \neq j_s$, we have $\lim_k \|E_{M_1}(x b_k y)\|_2 = 0$. For all $k$, we have
\begin{align*}
E_{M_1}(x b_k y) &= E_{M_1}(x_r \cdots x_1 \, b_k \, y_1 \cdots y_s) \\
& = E_{M_1}(x_r \cdots x_2 \, E_B( x_1 b_k  y_1) \, y_2 \cdots y_s) \\
& = E_{M_1}(x_r \cdots x_3 \, E_B (x_2  E_B( x_1 b_k  y_1)  y_2) \, y_3  \cdots y_s).
\end{align*}
Since $E_B( x_1 b_k  y_1) \to 0$ weakly as $k \to \infty$ and since the inclusion $B \subset M_2$ is mixing, we have $\lim_k \|E_B (x_2  E_B( x_1 b_k  y_1)  y_2)\|_2 = 0$ and hence $\lim_k \|E_{M_1}(x b_k y)\|_2 = 0$.
\end{proof}

\begin{proof}[Proof of Proposition $\ref{mixing-inclusion}$]
Put $M_1 = \Gamma(H_\R \ominus K_\R)\dpr \rtimes_{\pi_{H \ominus K}} G$, $M_2 = \Gamma(K_\R)\dpr \rtimes_{\pi_{K}} G$, $M = \Gamma(H_\R)\dpr \rtimes_{\pi_{H}} G$ and $B = \LL(G)$. Observe that $M = M_1 \ast_B M_2$. Since $\pi_K : G \to \mathcal O(K_\R)$ is mixing, the inclusion $B \subset M_2$ is mixing and so is the inclusion $M_1 \subset M$ by Proposition~\ref{mixing-afp}.
\end{proof}

\begin{cor}\label{mixing-consequence}
Let $G$ be any countable discrete group and $\pi : G \to \mathcal O(H_\R)$ any orthogonal representation. Let $K_\R \subset H_\R$ be a nonzero closed $\pi(G)$-invariant subspace such that $\pi | K_\R$ is mixing. Put $\pi_{H \ominus K} = \pi | H_\R \ominus K_\R$, $M = \Gamma(H_\R)\dpr \rtimes_\pi G$ and $N = \Gamma(H_\R \ominus K_\R)\dpr \rtimes_{\pi_{H \ominus K}}G$.

Let $e \in M$ be a nonzero projection and $A \subset eMe$ a diffuse subalgebra such that $A \preceq_M N$. Then $\QN_{eMe}(A)\dpr \preceq_M N$.
\end{cor}

\begin{proof}
Since $A \preceq_M N$, there exist $k \geq 1$, a projection $p \in \mathbf M_k(N)$, a nonzero partial isometry $v \in \mathbf M_{1, k}(e M) p$ and a unital $\ast$-homomorphism $\varphi : A \to p \mathbf M_k(N) p$ such that $a v = v \varphi(a)$ for all $a \in A$. Observe that $vv^* \in A' \cap eMe$ and $v^*v \in \varphi(A)' \cap p \mathbf M_k(M) p$. We moreover have 
$$v^* \QN_{e M e}(A)\dpr v \subset \QN_{v^*v \mathbf M_k(M) v^*v}(\varphi(A) v^*v)\dpr = v^*v \QN_{p \mathbf M_k(M) p}(\varphi(A))\dpr v^*v.$$

Since the inclusion $N \subset M$ is mixing by Proposition \ref{mixing-inclusion}, so is the inclusion $p \mathbf M_k(N) p \subset p \mathbf M_k(M) p$. Since $\varphi(A)$ is diffuse, the inclusion $p \mathbf M_k(N) p \subset p \mathbf M_k(M) p$ is weakly mixing through $\varphi(A)$.
By \cite[Proposition 6.14]{popa-vaes-advances}, we get $v^*v \in \mathbf M_k(N)$ and 
$$v^*v \QN_{p \mathbf M_k(M) p}(\varphi(A))\dpr v^*v \subset v^*v \mathbf M_k(N) v^*v.$$
Therefore, we have $v^* \QN_{e M e}(A)\dpr v \subset v^*v \mathbf M_k(N) v^*v$, whence $\QN_{eMe}(A)\dpr \preceq_M N$.
\end{proof}

\section{Relative asymptotic orthogonality property}\label{AOP}

In his seminal article \cite{popa-amenable}, Popa proved that the generator masa $A \subset M$ in a free group factor $M = \LL(\F_n)$ $(n \geq 2)$ satisfies the {\em asymptotic orthogonality property}, that is, for all $x, y \in (M^\omega \ominus A^\omega) \cap A'$ and all $a, b \in M \ominus A$, the vectors $a x$ and $y b$ are orthogonal in $\LL^2(M^\omega)$ (see \cite[Lemma 2.1]{popa-amenable}). He then used this property to deduce that the generator masa is maximal amenable inside the free group factor (see \cite[Corollary 3.3]{popa-amenable}).

We will need the following {\em relative} notion of asymptotic orthogonaliy property.

\begin{df}
Let $A \subset N \subset (M, \tau)$ be tracial von Neumann algebras. Let $\omega \in \beta(\N) \setminus \N$ be a free ultrafilter. We say that the inclusion $N \subset M$ has the {\em asymptotic orthogonality property relative to} $A$ if for all $x, y \in (M^\omega \ominus N^\omega) \cap A'$
and all $a, b \in M \ominus N$, the vectors $ax$ and $yb$ are orthogonal in $\LL^2(M^\omega)$.
\end{df}

The main technical result of this section is the following generalization of \cite[Theorem 3.2]{houdayer13}. In the initial version of the present paper, Theorem \ref{relative-AOP} was stated under the additional assumption that $G$ is abelian. I am very grateful to R\'emi Boutonnet for kindly pointing out to me that the proof of Theorem \ref{relative-AOP} could be slightly modified to show that Theorem \ref{relative-AOP} holds for any countable discrete group $G$.

Recall that an orthogonal representation $\pi : G \to \mathcal O(H_\R)$ is {\em compact} if $\pi$ is a direct sum of finite dimensional orthogonal representations.

\begin{theo}\label{relative-AOP}
Let $G$ be any countable discrete group and $\pi : G \to \mathcal O(H_\R)$ any orthogonal representation. Denote by $K_\R \subset H_\R$ the unique closed $\pi(G)$-invariant subspace such that $\pi_K = \pi | K_\R$ is weakly mixing and $\pi_{H \ominus K} = \pi | H_\R \ominus K_\R$ is compact. Put $M = \Gamma(H_\R)\dpr \rtimes_\pi G$ and $N = \Gamma(H_\R \ominus K_\R)\dpr \rtimes_{\pi_{H \ominus K}} G$. 

Then the inclusion $N \subset M$ has the asymptotic orthogonality property relative to $\LL(G)$.
\end{theo}

\begin{proof}
The proof is a further generalization of the proof of \cite[Theorem 3.2]{houdayer13}. Denote as usual by $H$ (resp.\ $K$) the complexified Hilbert space of $H_\R$ (resp.\ $K_\R$). The complex conjugation on $H$ is simply denoted by $e \mapsto \overline e$. The corresponding unitary representation will still be denoted by $\pi : G \to \mathcal U(H)$. The full Fock space of $H$ is defined by $\mathcal F(H) = \C \Omega \oplus \bigoplus_{n \geq 1} H^{\otimes n}$ and the Koopman representation of the free Bogoljubov action $\sigma_\pi : G \curvearrowright \Gamma(H_\R)\dpr$ is given by $\rho(g) = \id_{\C \Omega} \oplus \bigoplus_{n \geq 1} \pi(g)^{\otimes n}$ for all $g \in G$. 

Put $\pi_K = \pi | K_\R$, $\pi_{H \ominus K} = \pi | H_\R \ominus K_\R$, $M = \Gamma(H_\R)\dpr \rtimes_\pi G$ and $N = \Gamma(H_\R \ominus K_\R)\dpr \rtimes_{\pi_{H \ominus K}} G$. We may and will assume that $K_\R \neq 0$. We will identify $\LL^2(M)$ with $\mathcal F(H) \otimes \ell^2(G)$. Recall that the conjugation $\mathcal J : \mathcal F(H) \otimes \ell^2(G) \to \mathcal F(H) \otimes \ell^2(G)$ is defined by $\mathcal J \Omega =  \Omega$ and 
$$\mathcal J (e_1 \otimes \cdots \otimes e_n \otimes \delta_g) = \pi(g)^* \overline e_n \otimes \cdots \otimes \pi(g)^* \overline e_1 \otimes \delta_{g^{-1}}$$ 
for all $n \geq 1$, all $e_i \in H$ and all $g \in G$. 

Since the unitaries $(u_g)_{g \in G}$ implement the free Bogoljubov action $\sigma_\pi$, we also denote by $\rho : G \to \mathcal U(\LL^2(M))$ the unitary representation defined by $\rho(g) = u_g \, \mathcal J u_g \mathcal J$. 

We will be using the following notation throughout. Let $L \subset H$ be any closed subspace satisfying $L = \overline L$, that is, $L$ is stable under complex conjugation. 
\begin{itemize}
\item Denote by $\mathcal X(L)$ the closed linear span in $\mathcal F(H)$ of all the words $e_1 \otimes \cdots \otimes e_n$ of length $n \geq 1$ and such that $e_1 \in L$.

\item For $h \in G$, denote by $\mathcal Y_h(L)$ the closed linear span in $\mathcal F(H)$ of all the words $e_1 \otimes \cdots \otimes e_n$ of length $n \geq 1$ and such that $e_n \in \pi(h)L$.

\item Put $\mathscr X(L) = \mathcal X(L) \otimes \ell^2(G)$ and $\mathscr Y(L) = \bigoplus_{h \in G} (\mathcal Y_h(L) \otimes \C \delta_h)$. Observe that $\mathcal J \mathscr X(L) = \mathscr Y(L)$.
\end{itemize}

\begin{step}\label{step1}
Let $L \subset K$ be any finite dimensional subspace satisfying $L = \overline L$. Let $x = (x_n) \in (M^\omega \ominus N^\omega) \cap \LL(G)'$ and $w_1, w_2 \in \Gamma(H_\R \ominus K_\R)\dpr$ words of the following form: 
\begin{itemize}
\item $w_1 = 1$ or $w_1 = W(\zeta_1 \otimes \cdots \otimes \zeta_r)$ with $r \geq 1$ and letters $\zeta_i \in H \ominus K$. 
\item $w_2 = 1$ or $w_2 = W(\mu_1 \otimes \cdots \otimes \mu_s)$ with $s \geq 1$ and letters $\mu_j \in H \ominus K$. 
\end{itemize}
Then
$$\lim_{n \to \omega} \|P_{\mathscr X(L)}(w_1 x_n w_2)\|_2 = 0 \; \mbox{ and } \; \lim_{n \to \omega} \|P_{\mathscr Y(L)}(w_1 x_n w_2)\|_2 = 0.$$
\end{step}

\begin{proof}[Proof of Step $\ref{step1}$]
Observe that it suffices to show that $\lim_{n \to \omega} \|P_{\mathscr X(L)}(w_1 x_n w_2)\|_2 = 0$ for all $x = (x_n) \in (M^\omega \ominus N^\omega) \cap \LL(G)'$ and all words $w_1, w_2 \in \Gamma(H_\R \ominus K_\R)\dpr$ as in the statement. Indeed, assume that it is true. Then, we have
\begin{align*}
\lim_{n \to \omega} \|P_{\mathscr Y(L)}(w_1 x_n w_2)\|_2 &=  \lim_{n \to \omega} \|P_{\mathcal J \mathscr X(L)}(\mathcal J(w_2^* x_n^* w_1^*) )\|_2 \\
&= \lim_{n \to \omega} \|\mathcal J P_{\mathscr X(L)}(w_2^* x_n^* w_1^*)\|_2 \\
&= \lim_{n \to \omega} \| P_{\mathscr X(L)}(w_2^* x_n^* w_1^*)\|_2.
\end{align*}
Since $w_2^* = 1$ or $w_2^* = W(\overline \mu_s \otimes \cdots \otimes \overline \mu_1)$ and $w_1^* = 1$ or $w_1^* = W(\overline \zeta_r \otimes \cdots \otimes \overline \zeta_1)$ and since $(x_n^*) \in (M^\omega \ominus N^\omega) \cap \LL(G)'$, we will obtain $\lim_{n \to \omega} \|P_{\mathscr Y(L)}(w_1 x_n w_2)\|_2 = 0$.

Write $w_1 = W(\zeta_1 \otimes \cdots \otimes \zeta_r) \in N$ and $w_2 = W(\mu_1 \otimes \cdots \otimes \mu_s) \in N$ with $\zeta_i, \mu_j \in H \ominus K$. We will put $w_1 = 1$ if $r = 0$ and $w_2 = 1$ if $s = 0$ and we will put $w_1 = w_2 = 1$ if $K = H$. 
We may and will assume that $x = (x_n) \in (M^\omega \ominus N^\omega) \cap \LL(G)'$ satisfies $\sup_n \|x_n\|_\infty \leq 1$ and $x_n \in M \ominus N$ for all $n \in \N$. Write $x_n = \sum_{h \in G} (x_n)^h u_h$ for the Fourier expansion of $x_n \in M$ with respect to the crossed product decomposition $M = \Gamma(H_\R)\dpr \rtimes G$.

We use the following notation. Let $L \subset K$ be any closed subspace satisfying $L = \overline L$.
\begin{itemize}
\item Denote by $\mathcal H(r, L)$ the closed linear span in $\mathcal F(H)$ of all the words $e_1 \otimes \cdots \otimes e_n$ of length $n \geq r + 1$ and such that $e_1, \dots, e_r \in H \ominus K$ and $e_{r + 1} \in L$.

\item Put $\mathscr H(r, L) = \mathcal H(r, L) \otimes \ell^2(G)$.
\end{itemize}
By convention, we put $\mathcal H(r, L) = \mathcal X(L)$ if $r = 0$ or $K = H$. Observe that for all $g \in G$, $\rho(g) \mathcal H(r, L) = \mathcal H(r, \pi(g)L)$ and $\rho(g) \mathscr H(r, L) = \mathscr H(r, \pi(g)L)$. 

From now on, we assume that $L \subset K$ is finite dimensional and $L = \overline L$. We have $w_1 x_n w_2 = \sum_{h \in G} W(\zeta_1 \otimes \cdots \otimes \zeta_r) (x_n)^h W(\pi(h) \mu_1 \otimes \cdots \otimes \pi(h) \mu_s) \, u_h$. Then using repeatedly Proposition $\ref{wick}$ together with the facts that $x_n \in M \ominus N$ and $w_1, w_2 \in N$, we have 
\begin{equation}\label{simplification}
P_{\mathscr X(L)}(w_1 x_n w_2) = P_{\mathscr X(L)}(w_1 \mathcal J w_2^* \mathcal J P_{\mathscr H(r, L)}(x_n)).
\end{equation}

For all $n \in \N$ and all $g \in G$, we have
\begin{align}\label{popa-inequality1}
 \|\rho(g) P_{\mathscr H(r, L)} (x_n) \|_2^2 
& = \|\rho(g) P_{\mathscr H(r, L)} (x_n) - P_{\mathscr H(r, \pi(g) L)} (x_n) +  P_{\mathscr H(r, \pi(g) L)} (x_n)\|_2^2 \\ \nonumber
& \leq 2\|\rho(g) P_{\mathscr H(r, L)} (x_n) - P_{\mathscr H(r, \pi(g) L)} (x_n)\|_2^2 + 2\|P_{\mathscr H(r, \pi(g) L)} (x_n)\|_2^2 \\ \nonumber
& = 2\| P_{\mathscr H(r, \pi(g) L)}( u_{g} x_n u_{g}^* - x_n) \|_2^2 + 2\|P_{\mathscr H(r, \pi(g) L)} (x_n)\|_2^2 \\ \nonumber
& \leq 2\| u_g x_n u_g^* - x_n \|_2^2 + 2\|P_{\mathscr H(r, \pi(g) L)} (x_n)\|_2^2.
\end{align}

Fix $\ell \geq 1$. Choose $\varepsilon > 0$ very small such that  $\prod_{j = 0}^{\ell - 1} (1 + \delta^{\circ j}(\varepsilon))^2 \leq 2$ where 
$$\delta : \left[0, \frac12 \right) \to \R : t \mapsto \frac{2 t}{\sqrt{1 - t - \sqrt{2} \, t \sqrt{1 - t}}}$$
is the function which appeared in Section \ref{elementary}. Since $\pi_K$ is weakly mixing and $L \subset K$ is a finite dimensional subspace, by induction, we can find a sequence $e = g_1, \dots, g_{2^\ell}$ of pairwise distinct elements in $G$ with the property that 
$$\pi(g_{j}) L \perp_{\varepsilon/\dim L} \pi(g_{i}) L, \forall 1 \leq i < j \leq 2^\ell.$$
This yields 
\begin{equation}\label{perpendicular}
\mathscr H(r, \pi(g_j) L) \perp_{\varepsilon} \mathscr H(r, \pi(g_i) L), \forall 1 \leq i < j \leq 2^\ell.
\end{equation}

Indeed, this can be deduced from the following fact:

\begin{claim}
For all $g \in G$ and all $\varepsilon \geq 0$ such that $\pi(g) L \perp_{\varepsilon/\dim L} L$, we have $$\mathscr H(r, \pi(g) L) \perp_{\varepsilon} \mathscr H(r, L).$$
\end{claim}

\begin{proof}[Proof of the Claim]
Denote by $\mathcal H_r$ the closed linear span in $\mathcal F(H)$ of all the words $e_1 \otimes \cdots \otimes e_n$ of length $n \geq r$ and such that $e_1, \dots, e_r \in H \ominus K$. By convention, we put $\mathcal H_r = \mathcal F(H)$ if $r = 0$ or $K = H$.

Let $(e_i)_{i \geq 1}$ be an orthonormal basis for $H \ominus K$ and $(f_j)_{j \geq 1}$ an orthonormal basis for $K$ such that $(f_j)_{1 \leq j \leq \dim L}$ is an orthonormal basis for $L$. Define the unitary operator $U : K \otimes \mathcal H_r \otimes \ell^2(G) \to \mathscr H(r, K)$  by the formula
$$
U(f_j \otimes e_{i_1} \otimes \cdots \otimes e_{i_r} \otimes \xi \otimes \delta_h) = e_{i_1} \otimes \cdots \otimes e_{i_r} \otimes f_j \otimes \xi \otimes \delta_h.
$$
Observe that $U \rho(g) = \rho(g) U$ for all $g \in G$.

Let $g \in G$ and $\varepsilon \geq 0$ such that $\pi(g) L \perp_{\varepsilon/\dim L} L$. Let $\xi, \eta \in \mathscr H(r, L)$. Write $U^*\xi = \sum_{i = 1}^{\dim L} f_i \otimes \xi_i$ and $U^*\eta = \sum_{j = 1}^{\dim L} f_j \otimes \eta_j$ with $\xi_i , \eta_j \in \mathcal H_r \otimes \ell^2(G)$ such that $\|\xi\|^2 = \sum_{i = 1}^{\dim L} \|\xi_i\|^2$ and $\|\eta\|^2 = \sum_{j = 1}^{\dim L} \|\eta_j\|^2$. Using Cauchy-Schwarz inequality, we have
\begin{align*}
|\langle \rho(g) \xi, \eta \rangle| = |\langle U^* \rho(g) \xi, U^* \eta \rangle| = |\langle \rho(g) U^*\xi, U^*\eta\rangle| & \leq \sum_{i, j = 1}^{\dim L} |\langle \pi(g) f_i, f_j \rangle| |\langle \rho(g) \xi_i, \eta_j\rangle| \\ &  \leq \frac{\varepsilon}{\dim L} \sum_{i, j = 1}^{\dim L}\|\xi_i\| \|\eta_j\| \\
& \leq \varepsilon \|\xi\| \|\eta\|.
\end{align*}
This shows that $\rho(g) \mathscr H(r, L) \perp_{\varepsilon} \mathscr H(r, L)$, that is, $\mathscr H(r, \pi(g) L) \perp_{\varepsilon} \mathscr H(r, L)$. 
\end{proof}

Therefore, using Proposition \ref{projections} and the above $(\ref{popa-inequality1})$ and $(\ref{perpendicular})$, for all $n \in \N$, we get 
\begin{align*}
2^\ell \| P_{\mathscr H(r, L)} (x_n) \|_2^2  &= \sum_{i = 1}^{2^\ell} \| \rho(g_i) P_{\mathscr H(r, L)} (x_n) \|_2^2 \\
& \leq  \sum_{i = 1}^{2^\ell} \left( 2\| u_{g_i} x_n u_{g_i}^* - x_n \|_2^2 + 2\| P_{\mathscr H(r, \pi(g_i) L)} (x_n) \|_2^2 \right) \\
& \leq  2 \sum_{i = 1}^{2^\ell} \| u_{g_i} x_n u_{g_i}^* - x_n \|_2^2 + 2 \prod_{j = 0}^{\ell - 1} (1 + \delta^{\circ j}(\varepsilon))^2 \|x_n\|_2^2 \\
& \leq  2 \sum_{i = 1}^{2^\ell} \| u_{g_i} x_n u_{g_i}^* - x_n \|_2^2 + 4 \|x_n\|_2^2.
\end{align*}

This yields $\lim_{n \to \omega} \|P_{\mathscr H(r, L)} (x_n)\|_2^2  \leq 2^{2 - \ell}$. Since this is true for every $\ell \geq 1$, we finally get $\lim_{n \to \omega} \|P_{\mathscr H(r, L)} (x_n)\|_2 = 0$. Therefore, $\lim_{n \to \omega} \|P_{\mathscr X(L)}(w_1 x_n w_2)\|_2 = 0$ by $(\ref{simplification})$. This finishes the proof of Step $\ref{step1}$.
\end{proof}

\begin{step}\label{step2}
The inclusion $N \subset M$ has the asymptotic orthogonality property relative to $\LL(G)$.
\end{step}

\begin{proof}[Proof of Step $\ref{step2}$]

Observe that in order to show that $N \subset M$ has the asymptotic orthogonality property relative to $\LL(G)$, using a standard density argument together with Proposition \ref{wick}, it suffices to show that $a x \perp y b$ in $\LL^2(M^\omega)$ for all $x, y \in (M^\omega \ominus N^\omega) \cap \LL(G)'$ and all $a, b \in M \ominus N$ of the form $a = w_1 \, W(\xi_1 \otimes \cdots \otimes \xi_r) u_g \, w_2$ and $b = w_3 \, W(\eta_1 \otimes \cdots \otimes \eta_s) \, w_4$ with $w_1, w_2, w_3, w_4$ words in $\Gamma(H_\R \ominus K_\R)\dpr$ as in the statement of Step \ref{step1}, $r, s \geq 1$, $\xi_i, \eta_j$ letters in $K$ or $H \ominus K$, $\xi_1, \xi_r, \eta_1, \eta_s \in K$ and $g \in G$. There are two cases to consider:

$(1)$ Assume first that $r \geq s$.

Denote by $L \subset K$ the smallest subspace containing $\xi_1, \pi(g)^*\xi_r, \eta_1, \eta_s \in K$ and satisfying $L = \overline L$. Note that $L$ is finite dimensional. 

For any closed subspaces $L_1, L_2 \subset K$ such that $L_1 = \overline L_1$ and $L_2 = \overline L_2$, denote by $\mathcal Z_h(L_1, L_2)$ the closed linear span in $\mathcal F(H)$ of all the words $e_1 \otimes \cdots \otimes e_n$ of length $n \geq 2$ and such that $e_1 \in H \ominus L_1$ and $e_n \in H \ominus \pi(h) L_2$. Put $\mathscr Z(L_1, L_2) = \bigoplus_{h \in G} (\mathcal Z_h(L_1, L_2) \otimes \C \delta_h)$. 

We have
\begin{align*}
\langle a x, y b\rangle_{\LL^2(M^\omega)} & = \lim_{n \to \omega} \langle w_1 W(\xi_1 \otimes \cdots \otimes \xi_r)u_g w_2 \, x_n, y_n \, w_3 W(\eta_1 \otimes \cdots \otimes \eta_s) w_4 \rangle_{\LL^2(M)} \\
& = \lim_{n \to \omega} \langle W(\xi_1 \otimes \cdots \otimes \xi_r)u_g \, w_2 x_n w_4^* , w_1^* y_n w_3 \, W(\eta_1 \otimes \cdots \otimes \eta_s) \rangle_{\LL^2(M)}.
\end{align*}

Since $L \subset K$ is finite dimensional aand since $w_2 x_n w_4^*, w_1^* y_n w_3 \in M \ominus N$, Step $\ref{step1}$ implies that 
\begin{align*}
\lim_{n \to \omega} \|w_2 x_n w_4^* - P_{\left( (H \ominus L) \otimes \ell^2(G) \right) \oplus \mathscr Z(L, L)}(w_2 x_n w^*_4)\|_2 &= 0 \\
\lim_{n \to \omega} \|w_1^* y_n w_3 - P_{\left( (H \ominus L) \otimes \ell^2(G) \right) \oplus \mathscr Z(L, L)}(w_1^* y_n w_3)\|_2 &= 0.
\end{align*}
Observe that $u_g \left( (H \ominus L) \otimes \ell^2(G) \right) = (H \ominus \pi(g)L) \otimes \ell^2(G)$ and $u_g \mathscr Z(L, L) = \mathscr Z(\pi(g)L, L)$. 

Then Proposition $\ref{wick}$ and the definition of $L$ imply that 
\begin{align*}
W(\xi_1 \otimes \cdots \otimes \xi_r)u_g \left( \left( (H \ominus L) \otimes \ell^2(G) \right) \oplus \mathscr Z(L, L) \right) &\perp \mathcal J W(\overline \eta_s \otimes \cdots \otimes \overline \eta_1) \mathcal J \, \mathscr Z(L, L) \\
\mathcal J W(\overline \eta_s \otimes \cdots \otimes \overline \eta_1) \mathcal J \left( \left( (H \ominus L) \otimes \ell^2(G) \right) \oplus \mathscr Z(L, L) \right) &\perp W(\xi_1 \otimes \cdots \otimes \xi_r)u_g \, \mathscr Z(L, L).
\end{align*}

Therefore $\langle a x, y b\rangle_{\LL^2(M^\omega)}$ is equal to
$$\lim_{n \to \omega} \langle W(\xi_1 \otimes \cdots \otimes \xi_r)u_g \, P_{(H \ominus L) \otimes \ell^2(G)}(w_2 x_n w_4^*) , \mathcal J W(\overline \eta_s \otimes \cdots \otimes \overline \eta_1) \mathcal J P_{(H \ominus L) \otimes \ell^2(G)}(w_1^* y_n w_3) \rangle.$$
Since $r \geq s$, another application of Proposition $\ref{wick}$ yields
$$\mathcal J W( \eta_1 \otimes \cdots \otimes \eta_s) \mathcal J \, W(\xi_1 \otimes \cdots \otimes \xi_r)u_g \left ( (H \ominus L) \otimes \ell^2(G) \right) \perp (H \ominus L) \otimes \ell^2(G).$$
Therefore $\langle a x, y b\rangle_{\LL^2(M^\omega)} = 0$.

$(2)$ Assume now that $s \geq r + 1$. Denote by $\mathcal J^\omega$ the canonical conjugation on $\LL^2(M^\omega, \tau_\omega)$ defined by $\mathcal J^\omega v = v^*$ for all $v \in M^\omega$. We have
$$\langle ax, yb\rangle_{\LL^2(M^\omega)} = \langle \mathcal J^\omega (yb), \mathcal J^\omega(ax)\rangle_{\LL^2(M^\omega)} = \langle b^* y^*, x^* a^*\rangle_{\LL^2(M^\omega)}.$$
We have $x^*, y^* \in (M^\omega \ominus N^\omega) \cap \LL(G)'$, $b^* = w_4^* \, W(\overline \eta_s \otimes \cdots \otimes \overline \eta_1) \, w_3^*$ and $a^* = w_2^* \, u_g^* W(\overline \xi_r \otimes \cdots \otimes \overline \xi_1) \, w_1^*$. Put $c =  \sigma_\pi(g)(w_4^*) \, W(\pi(g) \overline \eta_s \otimes \cdots \otimes \pi(g) \overline \eta_1) u_g \, w_3^*$ and $d = \sigma_\pi(g)(w_2^*) \, W(\overline \xi_r \otimes \cdots \otimes \overline \xi_1) \, w_1^*$. We obtain
$$
\langle b^* y^*, x^* a^*\rangle_{\LL^2(M^\omega)} = \langle c y^*, x^* d\rangle_{\LL^2(M^\omega)}.
$$
By the first case, we get $\langle c y^*, x^* d\rangle_{\LL^2(M^\omega)} = 0$, whence $\langle ax, yb\rangle_{\LL^2(M^\omega)} = 0$.
\end{proof}
This finishes the proof of Theorem \ref{relative-AOP}.
\end{proof}

\section{Central sequences in ${\rm II_1}$ factors $\Gamma(H_\R)\dpr \rtimes_\pi G$}\label{gamma}

\subsection{Property Gamma}

The aim of this section is to prove Theorem \ref{thmA}. To do so, we first start by locating central sequences in $\Gamma(H_\R)\dpr \rtimes_\pi G$: when $\dim(H_\R) = \infty$, any central sequence in $\Gamma(H_\R)\dpr \rtimes_\pi G$ must asymptotically lie in $\LL(G)$.

\begin{prop}\label{commutant}
Let $G$ be any countable discrete group and $\pi : G \to \mathcal{O}(H_\R)$ any infinite dimensional orthogonal representation.  Put $M = \Gamma(H_\R)\dpr \rtimes_\pi G$. Then for every free ultrafilter $\omega \in \beta(\N) \setminus \N$, we have $M' \cap M^\omega \subset \LL(G)^\omega$.
\end{prop}

\begin{proof}
There are two cases to consider. Assume first that the representation $\pi$ is {\em reducible} and write $\pi = \pi_1 \oplus \pi_2$ and $H_\R = H_\R^{(1)} \oplus H_\R^{(2)}$. Then we have that $M$ can be written as the amalgamated free product
$$M = \left( \Gamma(H_\R^{(1)})\dpr \rtimes_{\pi_1} G \right) \ast_{\LL(G)} \left( \Gamma(H_\R^{(2)})\dpr \rtimes_{\pi_2} G\right).$$
Since $\dim \pi_i \geq 1$, we have that $\Gamma(H_\R^{(i)})\dpr$ is diffuse. An application of \cite[Lemma 6.1]{ioana-cartan} yields $M' \cap M^\omega \subset \LL(G)^\omega$.

Assume now that the representation $\pi$ is {\em irreducible}. Since $\pi$ is also infinite dimensional, it follows that $\pi$ is weakly mixing. We keep the same notation as in the proof of Theorem \ref{relative-AOP}.

Let $x = (x_n) \in M' \cap M^\omega$ and write $y = x - E_{\LL(G)^\omega}(x)$. Observe that $y = (y_n) \in (M^\omega \ominus \LL(G)^\omega) \cap \LL(G)'$ with $y_n = x_n - E_{\LL(G)}(x_n)$. For any closed subspace $L \subset K$ that is closed under complex conjugation and any $r \geq 1$, we denote by $\mathcal X_r(L)$ the closed linear span in $\mathcal F(H)$ of all the words $e_1 \otimes \cdots \otimes e_n$ of length $n \geq r$ and such that $e_1 \in L$.

Fix a nonzero vector $\xi \in H$. We have 
\begin{equation}\label{eq1}
\lim_{n \to \omega} \|y_n - P_{\mathcal X_1( H \ominus \C \xi) \otimes \ell^2(G)}(y_n)\|_2 = 0
\end{equation}
by Step $1$ in Theorem \ref{relative-AOP}. Using Proposition $\ref{wick}$, we have 
\begin{align*}
W(\xi) \left( \mathcal X_1(H \ominus \C \xi) \otimes \ell^2(G) \right) &\subset \mathcal X_2(\C \xi) \otimes \ell^2(G) \\
\mathcal J W(\overline \xi) \mathcal J \left( \mathcal X_1(H \ominus \C \xi) \otimes \ell^2(G) \right) &\subset \left( \C \Omega \oplus \mathcal X_1(H \ominus \C \xi) \right) \otimes \ell^2(G).
\end{align*}

In particular, we get 
\begin{align}\label{eq2}
W(\xi) \left( \mathcal X_1(H \ominus \C \xi) \otimes \ell^2(G) \right) &\perp H \otimes \ell^2(G) \\ \nonumber
W(\xi) \left( \mathcal X_1(H \ominus \C \xi) \otimes \ell^2(G) \right) &\perp \mathcal J W(\overline \xi) \mathcal J \left( \mathcal X_1(H \ominus \C \xi) \otimes \ell^2(G) \right).
\end{align}

For all $n \in \N$, we have 
\begin{align*}
W(\xi) x_n - x_n W(\xi) &= W(\xi) (E_{\LL(G)}(x_n) + y_n) - \mathcal J W(\overline \xi) \mathcal J  (E_{\LL(G)}(x_n) + y_n)\\
&= \left( W(\xi) E_{\LL(G)}(x_n) - \mathcal J W(\overline \xi) \mathcal J E_{\LL(G)}(x_n) - \mathcal J W(\overline \xi) \mathcal J y_n \right) + W(\xi) y_n
\end{align*}
Since $\lim_{n \to \omega} \|W(\xi) x_n - x_n W(\xi)\|_2 = 0$, a combination of $(\ref{eq1})$ and $(\ref{eq2})$ yields 
\begin{equation}\label{eq3}
\lim_{n \to \omega} \|W(\xi) y_n\|_2 = 0 \mbox{ and } \lim_{n \to \omega} \|W(\xi) E_{\LL(G)}(x_n) - \mathcal J W(\overline \xi) \mathcal J E_{\LL(G)}(x_n) - \mathcal J W(\overline \xi) \mathcal J y_n \|_2 = 0.
\end{equation}

Proposition \ref{wick} yields $\|W(\xi) P_{\mathcal X_1(H \ominus \C \xi) \otimes \ell^2(G)}(y_n)\|_2 = \|\xi\| \|P_{\mathcal X_1(H \ominus \C \xi) \otimes \ell^2(G)}(y_n)\|_2$. By $(\ref{eq1})$ and $(\ref{eq3})$, we get $\lim_{n \to \omega} \|y_n\|_2 = 0$, whence $\lim_{n \to \omega} \|x_n - E_{\LL(G)}(x_n)\|_2 = 0$. This shows that $M' \cap M^\omega \subset \LL(G)^\omega$ and finishes the proof of Proposition \ref{commutant}.
\end{proof}

\begin{proof}[Proof of Theorem \ref{thmA}]
Assume first that $\dim(H_\R) < \infty$. Since $\mathcal O(H_\R)$ is a compact group and $\pi(G)$ is discrete in $\mathcal O(H_\R)$, it follows that $\pi(G)$ is finite, whence $G$ is finite since $\pi$ is faithful. Then $\Gamma(H_\R)\dpr \subset \Gamma(H_\R)\dpr \rtimes_\pi G$ is a finite index inclusion of ${\rm II_1}$ factors. Since $\Gamma(H_\R)\dpr$ does not have property Gamma, $\Gamma(H_\R)\dpr \rtimes_\pi G$ does not have property Gamma either by \cite[Proposition 1.11]{pimsner-popa}.

Assume now that $\dim(H_\R) = \infty$ and put $M = \Gamma(H_\R)\dpr \rtimes_\pi G$. Let $x = (x_n) \in M' \cap M^\omega$. Since $M' \cap M^\omega \subset \LL(G)^\omega$ by Proposition $\ref{commutant}$, we may assume that $x_n \in \LL(G)$ for all $n \in \N$ and $\lim_{n \to \omega} \|y x_n - x_n y\|_2 = 0$ for all $y \in M$. Observe that since $\pi(G)$ is discrete in $\mathcal O(H_\R)$, we have that $\sigma_\pi(G)$ is discrete in $\Aut(\Gamma(H_\R)\dpr)$. Moreover, since $\pi : G \to \mathcal O(H_\R)$ is faithful, we have that $\sigma_\pi : G \to \Aut(\Gamma(H_\R)\dpr)$ is faithful.

Therefore, there exist $\kappa > 0$ and $y_1, \dots, y_k \in \Gamma(H_\R)\dpr$ such that the open neighborhood of $\id$ in $\Aut(\Gamma(H_\R)\dpr)$ defined by $\mathcal V (y_1, \dots, y_k, \kappa) = \{\theta \in \Aut(\Gamma(H_\R)\dpr) : \sum_{i = 1}^k \|\theta(y_i) - y_i\|_2^2 < \kappa\}$ satisfies $\sigma_\pi(G) \cap \mathcal V(y_1, \dots, y_k, \kappa) = \{\id\}$. Thus, we have
$$\sum_{i = 1}^k \|\sigma_\pi(g)(y_i) - y_i\|_2^2 \geq \kappa, \forall g \in G \setminus \{e\}.$$
Write $x_n = \sum_{g \in G} (x_n)^g u_g$ for the Fourier expansion of $x_n$ in $\LL(G)$. We have
\begin{align*}
\sum_{i = 1}^k \|y_i x_n - x_n y_i \|_2^2 &= \sum_{i = 1}^k \sum_{g \in G \setminus \{e\}} |(x_n)^g|^2 \|y_i - \sigma_\pi(g)(y_i)\|_2^2 \\
&= \sum_{g \in G \setminus \{e\}} |(x_n)^g|^2 \sum_{i = 1}^k \|y_i - \sigma_\pi(g)(y_i)\|_2^2 \\
&\geq \kappa \sum_{g \in G \setminus \{e\}} |(x_n)^g|^2 = \kappa \|x_n - \tau(x_n) 1\|_2^2.
\end{align*}
Since $\lim_{n \to \omega} \sum_{i = 1}^k \|y_i x_n - x_n y_i \|_2^2 = 0$, we get $\lim_{n \to \omega} \|x_n - \tau(x_n) 1\|_2 = 0$. Therefore $M$ does not have property Gamma.
\end{proof}

\begin{cor}\label{equivalence-gamma}
Let $G$ be any abelian countable discrete group and $\pi : G \to \mathcal O(H_\R)$ any faithful orthogonal representation such that $\dim H_\R \geq 2$.

Then $\Gamma(H_\R)\dpr \rtimes_\pi G$ is a ${\rm II_1}$ factor which does not have property Gamma if and only if $\pi(G)$ is discrete in $\mathcal O(H_\R)$ with respect to the strong topology.
\end{cor}

\begin{proof}
Assume $\pi(G)$ is not discrete in $\mathcal O(H_\R)$ with respect to the strong topology. Let $g_n \in G \setminus \{e\}$ be a sequence such that $\pi(g_n) \to 1$ strongly. Then $\sigma_\pi(g_n) \to \id$ in $\Aut(\Gamma(H_\R)\dpr)$, that is, $\lim_n \|u_{g_n} x u_{g_n}^* - x\|_2 = 0$ for all $x \in \Gamma(H_\R)\dpr$. Since $G$ is abelian, $(u_{g_n})$ is a central sequence in $\Gamma(H_\R)\dpr \rtimes_\pi G$ with $\tau(u_{g_n}) = 0$ for all $n \in \N$. Therefore, $\Gamma(H_\R)\dpr \rtimes_\pi G$ has property Gamma.
\end{proof}

Observe that whenever a faithful orthogonal representation $\pi : G \to \mathcal O(H_\R)$ contains a mixing subrepresentation, then $\pi(G)$ is discrete in $\mathcal O(H_\R)$. Indeed, let $K_\R \subset H_\R$ be a nonzero closed $\pi(G)$-invariant subspace such that $\pi | K_\R$ is mixing. Let $(g_n)_n$ be a sequence in $G$ such that $\pi(g_n) \to 1$ strongly. We have in particular $\lim_n \| \pi(g_n) \xi - \xi \| = 0$ for all $\xi \in K_\R$. We claim that $\{ g_n : n \in \N \}$ is finite. Otherwise, we can find a subsequence $(g_{n_k})_k$ such that $g_{n_k} \to \infty$ in $G$. By the mixing property of $\pi$, we get $\lim_k \|\pi(g_{n_k})\xi - \xi \| = \sqrt{2} \|\xi\|$ for all $\xi \in K_\R$, which is a contradiction. Since $\{g_n : n \in \N\}$ is finite and $\pi(g_n) \to 1$ strongly and $\pi$ is faithful, we obtain that $g_n = e$ for $n \in \N$ large enough.

Proposition \ref{commutant} provides another sufficient condition which ensures that the crossed product ${\rm II_1}$ factor $\Gamma(H_\R)\dpr \rtimes_\pi G$ does not have property Gamma.

\begin{cor}\label{sufficient-nongamma-bis}
Let $G$ be any countable discrete group such that $\LL(G)$ does not have property Gamma and $\pi : G \to \mathcal O(H_\R)$ any infinite dimensional orthogonal representation. Then $\Gamma(H_\R)\dpr \rtimes_\pi G$ does not have property Gamma.
\end{cor} 

\begin{proof}
Put $M = \Gamma(H_\R)\dpr \rtimes_\pi G$. By Proposition \ref{commutant}, we have
$$M' \cap M^\omega = M' \cap \LL(G)^\omega \subset \LL(G)' \cap \LL(G)^\omega.$$
If $M$ has property Gamma then $M' \cap M^\omega$ is diffuse and so is $\LL(G)' \cap \LL(G)^\omega$. Therefore, $\LL(G)$ has property Gamma.
\end{proof}

\begin{rem}
In case the group $G$ is not inner amenable, Corollary \ref{sufficient-nongamma-bis} is a particular case of a more general phenomenon. Indeed, any trace-preserving action $G \curvearrowright Q$ of such a group $G$ on a ${\rm II_1}$ factor $Q$ which does not have property Gamma gives rise to a crossed product ${\rm II_1}$ factor $Q \rtimes G$ which does not have property Gamma either (see \cite[Corollary]{choda-full}).

We mention that recently, Vaes \cite{vaes-inner-amenable} discovered an example of an inner amenable group $G$ with infinite conjugacy classes for which $\LL(G)$ does not have property Gamma.
\end{rem}

\subsection{Spectral gap rigidity}

In this section, we use Popa's spectral gap rigidity principle \cite{popasup} (see also \cite[Theorem 4.3]{peterson-L2}) to prove that the malleable deformation $(\theta_t)$ introduced in Section \ref{deformation} must converge uniformly on the unit ball of the relative commutant of {\em large} subalgebras of $\Gamma(H_\R)\dpr \rtimes_\pi G$.

We keep the same notation as in Section \ref{deformation}. For every $t \in \R$, denote by $(\theta_t^\omega)$ the unique one-parameter family of $\ast$-isomorphisms $\theta_t^\omega : M^\omega \to \widetilde M^\omega$ such that $\theta_t^\omega((x_n)) = (\theta_t(x_n))$ for all $(x_n) \in M^\omega$. Note however that the map $\R \to \Aut(\widetilde M^\omega) : t \mapsto \theta_t^\omega$ is quite discontinuous in general.

\begin{theo}\label{uniform-spectral-gap}
Let $G$ be any countable discrete group and $\pi : G \to \mathcal{O}(H_\R)$ any orthogonal representation. Put $M = \Gamma(H_\R)\dpr \rtimes_\pi G$. Let $p \in M$ be a nonzero projection. Let $P \subset pMp$ be a von Neumann subalgebra. Then at least one of the following conclusions holds true:
\begin{itemize}
\item There exists a nonzero projection $z \in \mathcal Z(P' \cap pMp)$ such that $Pz$ is amenable relative to $\LL(G)$ inside $M$.
\item The deformation $(\theta_t^\omega)$ converges uniformly to $\id$ in $\|\cdot\|_2$ on $(P' \cap p M^\omega p)_1$.
\end{itemize}
\end{theo}

\begin{proof}
Assume that the deformation $(\theta_t^\omega)$ does not converge uniformly to $\id$ in $\|\cdot\|_2$ on $(P' \cap p M^\omega p)_1$. Then there exist $c > 0$, a sequence $(t_k)$ of reals such that $\lim_{k \to \infty} t_k = 0$ and a sequence of elements $(y_k)$ in $(P' \cap pM^\omega p)_1$ such that $\inf_{k \in \N} \|y_k - \theta_{t_k}^\omega(y_k)\|_2 > c$. Write $y_k = (y_{k, n}) \in (P' \cap pM^\omega p)_1$ such that $\lim_{n \to \omega} \|b y_{k, n} - y_{k, n} b\|_2 = 0$ for all $b \in P$ and all $k \in \N$.

Let $I$ be the directed set of all $(\mathcal F, \varepsilon)$, with $\varepsilon >0$ and $\mathcal F \subset (P)_1$ finite subset. Let $i = (\mathcal F, \varepsilon) \in I$. Choose $k \in \N$ such that $\|a - \theta_{t_k}(a)\|_2 \leq \varepsilon/3$ for all $a \in \mathcal F$. Then choose $n \in \N$ such that $\|y_{k, n} - \theta_{t_k}(y_{k, n})\|_2 \geq c$ and $\| a y_{k, n} - y_{k, n} a\|_2 \leq \varepsilon/3$ for all $a \in \mathcal F$. 

Put $\xi_i := \theta_{t_k}(y_{k, n}) - (E_M \circ \theta_{t_k})(y_{k, n}) \in \LL^2(\widetilde M \ominus M)$. By Propostion $\ref{calculation}$, we have 
$$\|\xi_i\|_2 \geq \frac{1}{\sqrt{2}} \|y_{k, n} - \theta_{t_k}(y_{k, n})\|_2 \geq \frac{c}{\sqrt{2}}.$$
For all $x \in M$, we have
$$\|x \xi_i\|_2 = \|(1 - E_M) (x \theta_{t_k} (y_{k, n})) \|_2 \leq \|x \theta_{t_k}(y_{k, n})\|_2 \leq \|x\|_2.$$
By Popa's spectral gap argument \cite{popasup}, for all $a \in \mathcal F$, we have
\begin{align*}
\| a \xi_i - \xi_i a \|_2 &= \|(1 - E_M) (a \theta_{t_k}( y_{k, n} ) - \theta_{t_k}( y_{k, n}) a)\|_2 \leq \|a \theta_{t_k}( y_{k, n}) - \theta_{t_k}( y_{k, n}) a\|_2 \\
&\leq 2 \|a - \theta_{t_k}(a)\|_2 + \|a y_{k, n} - y_{k, n} a\|_2 \leq \varepsilon.
\end{align*}

Hence $\xi_i \in \LL^2(\widetilde M \ominus M)$ is a net of vectors satisfying $\limsup_i \|x \xi_i\|_2 \leq \|x\|_2$ for all $x \in M$, $\liminf_i \|\xi_i\|_2 \geq~\frac{c}{\sqrt{2}}$ and $\lim_i \|a \xi_i - \xi_i a\|_2 = 0$ for all $a \in P$. Since $\widetilde M = M \ast_{\LL(G)} \left( \Gamma(H_\R)\dpr \rtimes_\pi G\right)$, as $M$-$M$-bimodules, we have $\LL^2(\widetilde M \ominus M) \cong \LL^2(M) \otimes_{\LL(G)} \mathcal K$ for some $\LL(G)$-$M$-bimodule $\mathcal K$. By \cite[Lemma 2.3]{ioana-cartan}, there exists a nonzero projection $z \in \mathcal Z(P' \cap pMp)$ such that $Pz$ is amenable relative to $\LL(G)$ inside~ $M$.
\end{proof}

A straightforward combination of Theorems $\ref{intertwining1}$ and $\ref{uniform-spectral-gap}$ yields the following corollary:

\begin{cor}\label{corollary-spectral-gap}
Let $G$ be any countable discrete group and $\pi : G \to \mathcal{O}(H_\R)$ any orthogonal representation. Put $M = \Gamma(H_\R)\dpr \rtimes_\pi G$. Let $p \in M$ be a nonzero projection and $P \subset pMp$ a von Neumann subalgebra. Then at least one of the following conclusions holds true:
\begin{itemize}
\item There exists a nonzero projection $z \in \mathcal Z(P' \cap pMp)$ such that $Pz$ is amenable relative to $\LL(G)$ inside $M$.
\item $P' \cap pMp \preceq_M \LL(G)$.
\end{itemize}
\end{cor}

\section{Regular amenable subalgebras in ${\rm II_1}$ factors $\Gamma(H_\R)\dpr \rtimes_\pi G$}

The aim of this section is to prove Theorem \ref{thmB} and Corollary \ref{corC}.

\begin{proof}[Proof of Theorem \ref{thmB}]
Let $M = \Gamma(H_\R)\dpr \rtimes_\pi G$ and $p \in M$ a nonzero projection. Let $A \subset pMp$ be a von Neumann subalgebra that is amenable relative to $\LL(G)$ inside $M$ and denote $P =\mathcal N_{pMp}(A)\dpr$. Our aim is to show that $A \preceq_M \LL(G)$ or $P$ is amenable relative to $\LL(G)$ inside $M$. 

Put
$$\widetilde M = \Gamma(H_\R \oplus H_\R)\dpr \rtimes_{\pi \oplus \pi} G = \left( \Gamma(H_\R)\dpr \rtimes_{\pi} G \right) \ast_{\LL(G)} \left( \Gamma(H_\R)\dpr \rtimes_{\pi} G \right).$$
We identify $M$ with the left copy of $\Gamma(H_\R)\dpr \rtimes_{\pi} G$ and $\theta_1(M)$ with the right copy of $\Gamma(H_\R)\dpr \rtimes_{\pi} G$ inside the amalgamated free product $\widetilde M$. Note that we now use the malleable deformation $(\theta_t)$ from Section \ref{deformation}. Choose $t \in (0, 1)$ and put 
$$\mathcal A = \theta_t(A) \subset \theta_t(p) \widetilde M \theta_t(p) \;  \text{ and } \; \mathcal P = \mathcal N_{\theta_t(p)\widetilde M \theta_t(p)}(\mathcal A)\dpr.$$
Observe that $\theta_t(P) \subset \mathcal P$.

Since $A$ is amenable relative to $\LL(G)$ inside $M$ and since $M \subset \widetilde M$ is a tracial inclusion, it follows that $A$ is amenable relative to $\LL(G)$ inside $\widetilde M$. Since $\theta_t \in \Aut(\widetilde M)$ and $\theta_t (\LL(G)) = \LL(G)$, we get that $\mathcal A$ is amenable relative to $\LL(G)$ inside $\widetilde M$. By \cite[Theorem A]{Va13}, one of the following conditions holds true:
\begin{enumerate}
\item $\mathcal A \preceq_{\widetilde M} \LL(G)$.
\item $\mathcal P \preceq_{\widetilde M} M$ or $\mathcal P \preceq_{\widetilde M} \theta_1(M)$.
\item $\mathcal P$ is amenable relative to $\LL(G)$ inside $\widetilde M$.
\end{enumerate}

By Theorem $\ref{intertwining2}$, Condition $(1)$ leads to $A \preceq_M \LL(G)$. Observe that since $\theta_1 \in \Aut(\widetilde M)$, we have $\mathcal P \preceq_{\widetilde M} \theta_1(M)$ if and only if $\theta_{-1}(\mathcal P) \preceq_{\widetilde M} M$. So, Condition $(2)$ leads to $\theta_t(P) \preceq_{\widetilde M} M$ or $\theta_{t - 1}(P) \preceq_{\widetilde M} M$. Therefore by Theorem $\ref{intertwining2}$, Condition $(2)$ always leads to $P \preceq_M \LL(G)$ and hence $A \preceq_M \LL(G)$.

Finally assume that Condition $(3)$ holds. Since $\theta_t(P) \subset \mathcal P$ and $\theta_t(\LL(G)) = \LL(G)$, we have that $P$ is amenable relative to $\LL(G)$ inside $\widetilde M$. This means that the $p\widetilde M p$-$P$-bimodule $p\LL^2(\widetilde M)p$ is weakly contained in the $p \widetilde M p$-$P$-bimodule $p \LL^2(\widetilde M) \otimes_{\LL(G)} \overline{\LL^2(\widetilde M)} p$ and thus the $pMp$-$P$-bimodule $p\LL^2(M)p \subset p\LL^2(\widetilde M)p$ is weakly contained in the $pMp$-$P$-bimodule $p\LL^2(\widetilde M) \otimes_{\LL(G)} \overline{\LL^2(\widetilde M)}p$. Since as $M$-$M$-bimodules, we have $\LL^2(\widetilde M \ominus M) \cong \LL^2(M) \otimes_{\LL(G)} \mathcal K$ for some $\LL(G)$-$M$-bimodule $\mathcal K$, it follows that, as $pMp$-$pMp$-bimodules, we have
$$p \LL^2(\widetilde M) \otimes_{\LL(G)} \overline{\LL^2(\widetilde M)} p \cong p\LL^2(M) \otimes_{\LL(G)} \mathcal L.$$
for some $\LL(G)$-$P$-bimodule $\mathcal L$. Thus, the $pMp$-$P$-bimodule $p\LL^2(M)p$ is weakly contained in the $pMp$-$P$-bimodule $p\LL^2(M) \otimes_{\LL(G)} \mathcal L$. By \cite[Proposition 2.4(4)]{popa-vaes-cartan1}, it follows that $P$ is amenable relative to $\LL(G)$ inside $M$. This finishes the proof of Theorem \ref{thmB}.
\end{proof}

\begin{proof}[Proof of Corollary \ref{corC}]
Put $M = \Gamma(H_\R)\dpr \rtimes_\pi G$ and let $A \subset M$ be an amenable regular von Neumann subalgebra. Since $\dim H_\R \geq 2$, $\Gamma(H_\R)\dpr$ is a nonamenable ${\rm II_1}$ factor and so $M$ is not amenable relative to $\LL(G)$. Theorem \ref{thmB} implies that $A \preceq_M \LL(G)$.

$(1)$ Assume $\pi$ contains a direct sum of at least two finite dimensional subrepresentations. Write $\pi = \pi_1 \oplus \pi_2 \oplus \pi_3$, with $\pi_1$ and $\pi_2$ finite dimensional orthogonal representations. If $A \subset M$ is a Cartan subalgebra, we have $A \preceq_M \LL(G)$. Observe that for all $i \in \{1, 2\}$, since $\dim \pi_i$ is finite, the free Bogoljubov action $G \curvearrowright \Gamma(H_\R^{(i)})\dpr$ extends to a compact group action $\mathbf G \curvearrowright \Gamma(H_\R^{(i)})\dpr$. 

It follows from \cite[Corollary 4.2]{HLS} that for any trace-preserving action $\mathbf G \curvearrowright Q$ of a second countable compact group $\mathbf G$ on a nonamenable ${\rm II_1}$ factor $Q$ with separable predual, the fixed point algebra $Q^{\mathbf G}$ is necessarily diffuse. Since the free product of two diffuse von Neumann algebras is a nonamenable ${\rm II_1}$ factor, we get that $\LL(G)' \cap M$ has no amenable direct summand. Since $A \preceq_M \LL(G)$, we have $ \LL(G)' \cap M \preceq_M A' \cap M$ by \cite[Lemma 3.5]{vaes-bimodules}. However, since $A' \cap M = A$, this is a contradiction.

$(2)$ Assume $\pi$ contains a mixing subrepresentation, that is, let $K_\R \subset H_\R$ be a nonzero closed $\pi(G)$-invariant subspace such that $\pi | K_\R$ is mixing. Put $\pi_{H \ominus K} = \pi | H_\R \ominus K_\R$ and $N = \Gamma(H_\R \ominus K_\R)\dpr \rtimes_{\pi_{H \ominus K}} G$. If $A \subset M$ is a diffuse regular amenable subalgebra, we have $A \preceq_M \LL(G)$, whence $A \preceq_M N$. Since the inclusion $N \subset M$ is mixing by Proposition \ref{mixing-inclusion} and $M = \mathcal N_M(A)\dpr$, Corollary \ref{mixing-consequence} implies that $M \preceq_M N$. This means that $N p \subset pMp$ has finite index for some nonzero projection $p \in N' \cap M$. Since the inclusion $N \subset M$ is mixing, we moreover have $N' \cap M = \mathcal Z(N)$ by Corollary \ref{weak-mixing-consequence}, whence $p \in \mathcal Z(N)$. Since $N p \subset pMp$ has finite index, $N p$ is quasi-regular inside $pMp$ (see e.g.\ \cite[Definition/Proposition A.2]{vaes-bimodules}).

Since the inclusion $Np \subset pMp$ is mixing, we have $\QN_{pMp}(Np)\dpr = Np$ by Corollary \ref{weak-mixing-consequence}. Therefore, $Np = pMp$. Since $K_\R \neq 0$, the tracial von Neumann algebra $\Gamma(K_\R)\dpr$ is diffuse. Choose a Haar unitary $u \in \Gamma(K_\R)\dpr$. Since $\Gamma(K_\R)\dpr \ominus \C \subset M \ominus N$, we have $p u^k p = 0$ for all $k \in \Z \setminus \{0\}$, whence the projections $(u^k p u^{-k})_{k \in \Z}$ are pairwise orthogonal in $M$. Since $M$ is a tracial von Neumann algebra, we necessarily have $p = 0$. This is a contradiction and finishes the proof of Corollary \ref{corC}. 
\end{proof}

\section{Maximal amenable and maximal Gamma extensions}\label{maximal-extensions}

The aim of this section is to prove Theorems \ref{thmD} and \ref{thmE}. Theorem \ref{thmD} will be a consequence of the following more general result.

\begin{theo}\label{general-result}
Let $A \subset N \subset (M, \tau)$ be tracial von Neumann algebras such that $M$ has separable predual. Assume that the following conditions hold:
\begin{enumerate}
\item $A$ is amenable.
\item The inclusion $N \subset M$ is weakly mixing through $A$.
\item The inclusion $N \subset M$ has the asymptotic orthogonality property relative to $A$.
\end{enumerate}
Then for any  intermediate amenable von Neumann subalgebra $A \subset P \subset M$, we have $P \subset N$.
\end{theo}

\begin{proof}
Let $A \subset P \subset M$ be any intermediate amenable von Neumann subalgebra. Our aim is to show that in fact $P \subset N$. 

Since the inclusion $N \subset M$ is weakly mixing through $A$, we have $P' \cap M \subset A' \cap M \subset N$ by Corollary $\ref{weak-mixing-consequence}$, whence $P' \cap M = P' \cap N$. The set of projections $p \in P' \cap N$ with the property that $P p \subset pNp$, attains its maximum in a projection $z$ that belongs to $\mathcal Z(P' \cap N)$. Put $z^\perp = 1 - z \in \mathcal Z(P' \cap N)$. Our aim is to show that $z^\perp = 0$.

Assume by contradiction that $z^\perp \neq 0$. Put $Q = P z^\perp$. We first show that $Q \preceq_M N$. Assume by contradiction that $Q \npreceq_M N$. Since $Q$ is amenable and thus hyperfinite by Connes' result \cite{connes76}, we can write $Q = \bigvee_k Q_k$ where $(Q_k)_{k \geq 1}$ is an increasing sequence of unital finite dimensional $\ast$-subalgebras of $Q$ such that the inclusion $Q'_k \cap Q \subset Q$ has finite index for all $k \geq 1$. 

Indeed, let $q_n \in \mathcal Z(Q)$ be pairwise orthogonal central projections in $Q$ such that $\sum_{n \in \N} q_n = 1_Q = z^\perp$ and 
$$Q q_0 = \mathcal Z_0 \ovt R \; \mbox{ and } \; Q q_n = \mathcal Z_n \otimes \mathbf M_n(\C),$$
with $\mathcal Z_n$ an abelian von Neumann algebra for all $n \in \N$ and $R$ the unique hyperfinite ${\rm II_1}$ factor. So, $Q q_0$ is the direct summand of type ${\rm II_1}$ and $Q q_n$ is the homogeneous direct summand of type ${\rm I}_n$. For every $n \in \N$, let $(\mathcal Z_n^{(k)})_{k \geq 1}$ be an increasing sequence of unital finite dimensional $\ast$-subalgebras of $\mathcal Z_n$ such that $\mathcal Z_n = \bigvee_k \mathcal Z_n^{(k)}$. Regard $R = \ovt_{j = 1}^\infty (\mathbf M_2(\C), \tau_2)$ and put $R_k = \ovt_{j = 1}^k (\mathbf M_2(\C), \tau_2)$. 

For every $k \geq 1$, define the unital finite dimensional $\ast$-subalgebra $Q_k \subset Q$ by
$$Q_k = \left(\mathcal Z_0^{k} \otimes R_k \right) \oplus \bigoplus_{1 \leq n \leq k} \left( \mathcal Z_n^{(k)} \otimes \mathbf M_n(\C) \right) \oplus \C \sum_{n \geq k + 1} q_n.$$
Then $(Q_k)_{k \geq 1}$ is increasing, $\bigvee_k Q_k = Q$ and moreover
$$Q_k' \cap Q = \left(\mathcal Z_0 \ovt (R_k' \cap R) \right) \oplus \bigoplus_{1 \leq n \leq k} \left( \mathcal Z_n \otimes \C 1_{\mathbf M_n(\C)} \right) \oplus \bigoplus_{n \geq k + 1} \left( \mathcal Z_n \otimes \mathbf M_n(\C) \right).$$
Therefore, $Q_k' \cap Q \subset Q$ has finite index for all $k \geq 1$.

Since $Q \npreceq_M N$, we have $Q'_k \cap Q \npreceq_M N$ for all $k \geq 1$ by Remark \ref{remark-intertwining}. For every $k \geq 1$, choose $u_k \in \mathcal U(Q_k' \cap Q)$ such that $\|E_N(u_k)\|_2 \leq \frac1k \|z^\perp\|_2$. Put $u = (u_k) \in \mathcal U(Q' \cap Q^\omega)$ and observe that $u \in (M^\omega \ominus N^\omega) \cap P'$.

Since the inclusion $N \subset M$ has the asymptotic orthogonality property relative to $A$, we have $(y - E_N(y)) u \perp u (y - E_N(y))$ in $\LL^2(M^\omega)$ for all $y \in Q$. Since $yu = uy$ for all $y \in Q$, we get
\begin{equation}\label{equality}
\|E_N(y) u - u E_N(y)\|_2^2 = \|(y - E_N(y)) u\|_2^2 + \|u (y - E_N(y))\|_2^2 = 2 \|y - E_N(y)\|_2^2.
\end{equation}
Let $k \in \N$ large enough such that $\|E_N(u_k)\|_2 \leq \frac14 \|z^\perp\|_2$. We get $\|E_N(u_k) u - u E_N(u_k) \|_2 \leq \frac12 \|z^\perp\|_2$ and $\|u_k - E_N(u_k)\|_2 \geq \frac34 \|z^\perp\|_2$. This contradicts Equation $(\ref{equality})$.

Thus, we have $Q \preceq_M N$. There exist $k \geq 1$, a projection $p \in \mathbf M_k(N)$, a nonzero partial isometry $v \in \mathbf M_{1,k} (z^\perp M)p$ and a unital normal $\ast$-homomorphism $\varphi : Q \to p\mathbf M_k(N)p$ such that $a v = v \varphi(a)$ for all $a \in Q$. Write $v = [v_1 \cdots v_k] \in \mathbf M_{1,k} (z^\perp M)p$. In particular, we have $Q v_i \subset \sum_{j = 1}^k v_j N$ for all $1 \leq i \leq k$, whence $A v_i \subset \sum_{j = 1}^k v_j N$ for all $1 \leq i \leq k$. Since the inclusion $N \subset M$ is weakly mixing through $A$, we obtain that $v_i \in N$ for all $1 \leq i \leq k$ by Corollary \ref{weak-mixing-consequence}. Therefore $vv^* \in Q' \cap z^\perp N z^\perp$ and $Q vv^* \subset vv^* N vv^*$. We obtain $P(z + vv^*) \subset (z + vv^*) N (z + vv^*)$. This contradicts the fact that $z$ is the maximum projection $p \in P' \cap N$ with the property that $Pp \subset p N p$. Consequently, $z = 1$ and so $P \subset N$. 
\end{proof}

\begin{proof}[Proof of Theorem \ref{thmD}]
Put $M = \Gamma(H_\R)\dpr \rtimes_\pi G$ and $N = \Gamma(H_\R \ominus K_\R)\dpr \rtimes_{\pi_{H \ominus K}}G$. The inclusion $N \subset M$ is weakly mixing through $\LL(G)$ by Proposition \ref{weak-mixing-through} and has the asymptotic orthogonality property relative to $\LL(G)$ by Theorem \ref{relative-AOP}. Theorem \ref{thmD} is now a consequence of Theorem \ref{general-result}.
\end{proof}

\begin{proof}[Proof of Theorem \ref{thmE}]
Put $M = \Gamma(H_\R)\dpr \rtimes_\pi G$ and $N = \Gamma(H_\R \ominus K_\R)\dpr \rtimes_{\pi_{H \ominus K}}G$. Let $\LL(G) \subset P \subset M$ be any intermediate  von Neumann subalgebra with property Gamma. Our aim is to show that in fact $P \subset N$.

Since the inclusion $N \subset M$ is mixing by Proposition \ref{mixing-inclusion}, we have $P' \cap M \subset \LL(G)' \cap M \subset N$ by Corollary $\ref{weak-mixing-consequence}$, whence $P' \cap M = P' \cap N$. The set of projections $p \in P' \cap N$ with the property that $Pp$ is amenable attains its maximum in a projection $z$ that belongs to $\mathcal Z(\mathcal N_M(P)\dpr)$ and hence $ z \in \mathcal Z(P' \cap N)$ (see e.g.\ \cite[Lemma 2.6]{BV12}). Since the intermediate von Neumann subalgebra $\LL(G) \subset Pz \oplus \LL(G)z^\perp \subset M$ is amenable, we have $Pz \oplus \LL(G)z^\perp \subset N$ by Theorem \ref{thmD}, whence $Pz \subset zNz$. Put $z^\perp = 1 - z$. It remains to prove that $P z^\perp \subset z^\perp N z^\perp$.

The set of projections $p_0 \in (P z^\perp)' \cap z^\perp N z^\perp$ with the property that $(Pz^\perp)p_0 \subset p_0(z^\perp N z^\perp)p_0$, attains its maximum in a projection $z_0$ that belongs to $\mathcal Z((P z^\perp)' \cap z^\perp N z^\perp)$. Our aim is to show that $z_0 = z^\perp$. Put $q = z^\perp - z_0$ and observe that $q \in \mathcal Z(P' \cap N)$.

Assume by contradiction that $q \neq 0$. Let $\omega \in \beta(\N) \setminus \N$ be a free ultrafilter. Put $(Pq)_\omega = (P q)' \cap q M^\omega q$. Since $P' \cap P^\omega$ is diffuse, since $P' \cap P^\omega \subset P' \cap M^\omega$ and since $(Pq)_\omega = q(P' \cap M^\omega)q$, we get that $(Pq)_\omega$ is diffuse. There are two cases to consider:

$(1)$ Assume $(P q)_\omega \preceq_{M^\omega} \LL(G)^\omega$. Since $\LL(G)^\omega \subset N^\omega$ is a unital von Neumann subalgebra, we have $(P q)_\omega \preceq_{M^\omega} N^\omega$. Since the inclusion $N \subset M$ is mixing by Proposition \ref{mixing-inclusion} ans since $(P q)_\omega$ is diffuse, we get $Pq \preceq_M N$ by \cite[Lemma 9.5]{ioana-cartan}. 

$(2)$ Assume $(P q)_\omega \npreceq_{M^\omega} \LL(G)^\omega$. We now use an idea due to Peterson (see \cite[Theorem 4.5]{peterson-L2}). Recall that $(\theta_t)$ is the malleable deformation introduced in Section $\ref{deformation}$. Since $z$ is the maximum projection $p \in P' \cap M = P' \cap N$ such that $Pp$ is amenable and since $\LL(G)$ is amenable, we have that $(Pz^\perp)p$ is not amenable relative to $\LL(G)$ inside $M$ for all nonzero projection $p \in \mathcal Z((P z^\perp)' \cap z^\perp M z^\perp)$. Therefore, the deformation $(\theta_t^\omega)$ necessarily converges uniformly to $\id$ in $\|\cdot\|_2$ on $\mathcal U((P q)_\omega)$ by Theorem \ref{uniform-spectral-gap}. Let $\varepsilon > 0$. Choose $t > 0$ such that $\|v - \theta_t^\omega(v)\|_2 < \frac{\varepsilon^2}{8}$ for all $v \in \mathcal U((P q)_\omega)$.

Let $x \in (P q)_1$. Fix a $\|\cdot\|_2$ dense sequence $(y_i)_{i \geq 1}$ in $(q M)_1$. For every $n \geq 1$, there exists a unitary $v_n \in \mathcal U((P q)_\omega)$ such that $\|E_{\LL(G)^\omega}(y_i^* v_n y_j)\|_2 < \frac1n$ for all $1 \leq i, j \leq n$. Write $v_n = (v_{k, n}) \in \mathcal U((P q)_\omega)$ with $v_{k, n} \in \mathcal U(q M q)$ such that $\lim_{k \to \omega} \|v_{k, n} x - x v_{k, n}\|_2 = 0$ for all $n \geq 1$. Observe that $\|E_{\LL(G)^\omega}(y_i^* v_n y_j)\|_2 = \lim_{k \to \omega} \|E_{\LL(G)}(y_i^* v_{k, n} y_j)\|_2$ and $\|v_n - \theta_t^\omega(v_n)\|_2 = \lim_{k \to \omega} \|v_{k, n} - \theta_t(v_{k, n})\|_2$ for all $n \geq 1$. 

Thus, for all $n \geq 1$, there exists $k_n \in \N$ such that with $w_n = v_{n, k_n} \in \mathcal U(q M q)$, we have: 
\begin{itemize}
\item $\|w_n x - x w_n\|_2 \leq \frac1n$;
\item $\|E_{\LL(G)}(y_i^* w_n y_j)\|_2 \leq \frac1n$ for all $1 \leq i, j \leq n$;
\item $\|w_n - \theta_t(w_n)\|_2 \leq \frac{\varepsilon^2}{4}$.
\end{itemize}
Observe that since $(y_i)_{i \geq 1}$ is $\|\cdot\|_2$ dense in $(q M)_1$, we have that $\lim_n \|E_{\LL(G)}(c^* w_n d)\|_2 = 0$ for all $c, d \in q M$.

Put $\delta_t(y) = \theta_t(y) - (E_M \circ \theta_t)(y) \in \LL^2(\widetilde M \ominus M)$ for all $y \in M$. For all $n \geq 1$, we have
\begin{align}\label{mixing-delta}
\|\delta_t(x)\|_2^2 = \langle \delta_t(x), \delta_t(x)\rangle &\leq | \langle \delta_t(w_n x w_n^*), \delta_t(x)\rangle | + \|w_n x w_n^* - x\|_2 \\ \nonumber
& \leq | \langle w_n \delta_t(x) w_n^*, \delta_t(x)\rangle | + \|w_n x w_n^* - x\|_2 + 2 \|w_n - \theta_t(w_n)\|_2 \\ \nonumber
& \leq | \langle w_n \delta_t(x) w_n^*, \delta_t(x)\rangle | + \frac1n + \frac{\varepsilon^2}{2}.
\end{align}

\begin{claim}
Let $a_n \in (M)_1$ be a sequence such that $\lim_n \|E_{\LL(G)}(c^* a_n d)\|_2 = 0$ for all $c, d \in (M)_1$ and let $b_n \in (M)_1$ be any sequence. Then
$$\lim_n |\langle a_n \xi b_n, \eta \rangle| = 0, \forall \xi, \eta \in \LL^2(\widetilde M \ominus M).$$ 
\end{claim}

\begin{proof}[Proof of the Claim]
Recall that $\widetilde M = M \ast_{\LL(G)} \theta_1(M)$. It suffices to prove the Claim for $\xi, \eta \in \widetilde M \ominus M$ words of the form
$$\xi = x_1 x_2 \cdots x_{2k} x_{2k + 1} \; \mbox{ and } \; \eta = y_1 y_2 \cdots y_{2\ell} y_{2\ell + 1}$$
where $k, \ell \geq 1$; $x_1, x_{2k + 1}, y_1, y_{2 \ell + 1} \in M$; $x_{2i}, y_{2j} \in \theta_1(M) \ominus \LL(G)$ for all $1 \leq i \leq k$ and all $1 \leq j \leq \ell$; $x_{2i + 1}, y_{2j + 1} \in M \ominus \LL(G)$ for all $1 \leq i \leq k - 1$ and all $1 \leq j \leq \ell - 1$. We may moreover assume that 
$$\sup \{ \|x_{2i}\|_\infty, \|x_{2i \pm 1}\|_\infty, \|y_{2j}\|_\infty, \|y_{2j \pm 1}\|_\infty : 1 \leq i \leq k, 1 \leq j \leq \ell \} \leq 1.$$

Using the freeness with amalgamation over $\LL(G)$, we get
\begin{align*}
|\langle a_n \xi b_n, \eta \rangle| &= |\tau(y_{2\ell + 1}^* y_{2\ell}^* \cdots y_2^* \, y_1^* a_n x_1 \, x_2 \cdots x_{2k} x_{2k + 1} b_n)| \\
&= |\tau(y_{2\ell + 1}^* E_M(y_{2\ell}^* \cdots y_2^* \, y_1^* a_n x_1 \, x_2 \cdots x_{2k}) x_{2k + 1} b_n)| \\
&= |\tau(y_{2\ell + 1}^* E_M(y_{2\ell}^* \cdots y_2^* \, E_{\LL(G)}(y_1^* a_n x_1) \, x_2 \cdots x_{2k}) x_{2k + 1} b_n)| \\
&= |\tau(y_{2\ell + 1}^* y_{2\ell}^* \cdots y_2^* \, E_{\LL(G)}(y_1^* a_n x_1) \, x_2 \cdots x_{2k} x_{2k + 1} b_n)| \\
& \leq \|E_{\LL(G)}(y_1^* a_n x_1)\|_2.
\end{align*}
Therefore $\lim_n |\langle a_n \xi b_n, \eta \rangle| = 0$.
\end{proof}

Since $\lim_n \|E_{\LL(G)}(c^* w_n d)\|_2 = 0$ for all $c, d \in q M$ and since $\delta_t(x) \in \LL^2(\widetilde M \ominus M)$, the Claim yields $\lim_n | \langle w_n \delta_t(x) w_n^*, \delta_t(x)\rangle | = 0$. With the above inequality $(\ref{mixing-delta})$ and Proposition \ref{calculation}, we get
$$\|x - \theta_t(x)\|_2 \leq \sqrt{2} \|\delta_t(x)\|_2 \leq \varepsilon, \forall x \in (P q)_1.$$
By Theorem \ref{intertwining1}, we obtain $P q \preceq_M \LL(G)$, whence $P q \preceq_M N$.

Therefore, in both cases we obtain $Pq \preceq_M N$. Since the inclusion $N \subset M$ is mixing, the end of the proof of Theorem \ref{general-result} yields a nonzero projection $q_0 \in (Pq)' \cap  q N q$ such that $q_0 \leq q = z^\perp - z_0$ and $P q_0 \subset q_0 N q_0$. We obtain $(Pz^\perp)(z_0 + q_0) \subset (z_0 + q_0) (z^\perp N z^\perp) (z_0 + q_0)$. This contradicts the fact that $z_0$  is the maximum projection $p_0 \in (P z^\perp)' \cap  z^\perp N z^\perp$ with the property that $(P z^\perp) p_0 \subset p_0 (z^\perp  N z^\perp)  p_0$. Therefore, $z_0 = z^\perp$ and $Pz^\perp \subset  z^\perp N z^\perp$. This finally yields $P \subset N$ and finishes the proof of Theorem~\ref{thmE}.
\end{proof}

\section{Approximation properties for $\Gamma(H_\R)\dpr \rtimes_\pi G$}

\subsection{Complete bounded approximation property}

We refer to \cite[Chapter 12]{BO} for the notion of {\em weak amenability} for discrete groups $G$ and the definition of $\Lambda_{\cb}(G)$. 

Let $(M, \tau)$ be a tracial von Neumann algebra. Following \cite{cowling-haagerup}, we say that $M$ has the {\em completely bounded approximation property} if there exist $\kappa > 0$ and a net $\Phi_n : M \to M$ of normal finite rank (completely bounded) maps such that
\begin{enumerate}
\item $\lim_n \|\Phi_n(x) - x\|_2 = 0$ for all $x \in M$.
\item $\sup_n \|\Phi_n\|_{\cb} \leq \kappa$.
\end{enumerate}
The Cowling-Haagerup constant $\Lambda_{\cb}(M)$ is defined as the infimum of all values of $\kappa$ for which such nets exist. By \cite[Theorem 12.3.10]{BO}, we have that $\Lambda_{\cb}(\LL(G)) = \Lambda_{\cb}(G)$ for all countable discrete groups $G$.

\begin{theo}
Let $G$ be any countable discrete group and $\pi : G \to \mathcal O(H_\R)$ any compact orthogonal representation. Then $\Lambda_{\cb}(\Gamma(H_\R)\dpr \rtimes_\pi G) = \Lambda_{\cb}(G)$.
\end{theo}

\begin{proof}
We obviously have $\Lambda_{\cb}(\Gamma(H_\R)\dpr \rtimes_\pi G) \geq \Lambda_{\cb}(G)$. To prove the reverse inequality, we use techniques and results from \cite[Section 3]{houdayer-ricard}. We may and will assume that $\Lambda_{\cb}(G) < \infty$.

By \cite[Corollary 3.14]{houdayer-ricard}, there exists a sequence $\varphi_n : \N \to \C$ of finitely supported functions such that $\lim_n \varphi_n = 1$ pointwise and the corresponding unital trace-preserving {\em radial} multipliers ${\rm m}_{\varphi_n} : \Gamma(H_\R)\dpr \to \Gamma(H_\R)\dpr$ defined by 
$${\rm m}_{\varphi_n}(W(e_1 \otimes \cdots \otimes e_r)) = \varphi_n(r) W(e_1 \otimes \cdots \otimes e_r)$$
satisfy $\limsup_n \|{\rm m}_{\varphi_n}\|_{\cb} = 1$. Observe that since the radial multipliers ${\rm m}_{\varphi_n}$ commute with the free Bogoljubov action $\sigma_\pi$, we may extend ${\rm m}_{\varphi_n}$ to $\Gamma(H_\R)\dpr \rtimes G$ by the formula 
$${\rm m}_{\varphi_n}(W(e_1 \otimes \cdots \otimes e_r) u_g) = \varphi_n(r) W(e_1 \otimes \cdots \otimes e_r) u_g.$$
We still have $\limsup_n \|{\rm m}_{\varphi_n}\|_{\cb} = 1$.

Next, since $\pi$ is compact, we can write $\pi = \bigoplus_{j \in \N} \pi_j$ and $H_\R = \bigoplus_{j \in \N} H_\R^{(j)}$ with $\pi_j$ a finite dimensional orthogonal representation or $\pi_j = 0$. For $p \in \N$, let $E_p : H_\R \to \bigoplus_{0 \leq j \leq p} H_\R^{(j)}$ be the orthogonal projection and denote by $\Gamma(E_p) : \Gamma(H_\R)\dpr \to \Gamma(H_\R)\dpr$ the unique trace-preserving unital completely positive multiplier (see \cite[Section 2]{voiculescu92}) defined by
$$\Gamma(E_p)(W(e_1 \otimes \cdots \otimes e_r)) = W(E_p (e_1) \otimes \cdots \otimes E_p(e_r)).$$
Observe that since the completely positive multipliers $\Gamma(E_p)$ commute with the free Bogoljubov action $\sigma_\pi$, we may extend $\Gamma(E_p)$ to $\Gamma(H_\R)\dpr \rtimes G$ by the formula 
$$\Gamma(E_p)(W(e_1 \otimes \cdots \otimes e_r) u_g) = W(E_p(e_1) \otimes \cdots \otimes E_p(e_r)) u_g.$$

Let $\varepsilon > 0$. Since $\Lambda_{\cb}(G) < \infty$, let $\psi_q : G \to \C$ be a sequence of finitely supported functions such that $\psi_q(e) = 1$ for all $q$, $\lim_q \psi_q = 1$ pointwise and the corresponding unital trace-preserving Herz-Schur multipliers ${\rm m}_{\psi_q} : \LL(G) \to \LL(G)$ defined by ${\rm m}_{\psi_q}(u_g) = \psi_q(g) u_g$ satisfy $\sup_q \|{\rm m}_{\psi_q}\|_{\cb} \leq \Lambda_{\cb}(G) + \varepsilon$. We may extend ${\rm m}_{\psi_q}$ to $\Gamma(H_\R)\dpr \rtimes_\pi G$ by the formula 
$${\rm m}_{\psi_q}(W(e_1 \otimes \cdots \otimes e_r) u_g) = \psi_q(g) W(e_1 \otimes \cdots \otimes e_r) u_g.$$
We still have $\sup_q \|{\rm m}_{\psi_q}\|_{\cb} \leq \Lambda_{\cb}(G) + \varepsilon$. 

Define the trace-preserving unital finite rank (completely bounded) maps ${\rm M}_{n, p, q} : \Gamma(H_\R)\dpr \rtimes_\pi G \to \Gamma(H_\R)\dpr \rtimes_\pi G$ by the formula ${\rm M}_{n, p, q} = {\rm m}_{\varphi_n} \circ \Gamma(E_p) \circ {\rm m}_{\psi_q}$. We have $\lim_{n, p, q} \|{\rm M}_{n, p, q}(x) - x\|_2 = 0$ for all $x \in \Gamma(H_\R)\dpr \rtimes_\pi G$ and $\sup_{n \geq n_0, p, q} \|{\rm M}_{n, p, q}\|_{\cb} \leq \Lambda_{\cb}(G) + 2 \varepsilon$, for $n_0 \in \N$ sufficiently large. Since this is true for every $\varepsilon > 0$, we get $\Lambda_{\cb}(\Gamma(H_\R)\dpr \rtimes_\pi G) \leq \Lambda_{\cb}(G)$.
\end{proof}

\subsection{Relative Haagerup property}

Let $B \subset (M, \tau)$ be an inclusion of tracial von Neumann algebras. Whenever $\varphi : M \to M$ is a trace-preserving $B$-$B$-bimodular unital completely positive map, we denote $T_\varphi \in \langle M, e_B \rangle$ the unique bounded operator on $\LL^2(M)$ defined by $T_\varphi(x) = \varphi(x)$ for all $x \in M$. 

Following \cite[Definition 2.1]{Po01}, we say that $M$ has the {\em Haagerup property relative to} $B$ if there exists a net $\varphi_n : M \to M$ of trace-preserving $B$-$B$-bimodular unital completely positive maps such that 
\begin{enumerate}
\item $\lim_n \|\varphi_n(x) - x\|_2 = 0$ for all $x \in M$.
\item $\varphi_n$ is {\em compact over} $B$ for all $n$, that is, for all $\varepsilon > 0$, there exists a finite trace projection $p \in \langle M, e_B\rangle$ such that $\|T_{\varphi_n}(1 - p)\|_\infty \leq \varepsilon$.
\end{enumerate}
 When $M$ has the Haagerup property relative to $\C$, we simply say that $M$ has the {\em Haagerup property} (see \cite{choda}).

\begin{theo}\label{haagerup}
Let $G$ be any countable discrete group and $\pi : G \to \mathcal O(H_\R)$ any orthogonal representation. The following are equivalent:
\begin{enumerate}
\item $\pi$ is compact.
\item $\Gamma(H_\R)\dpr \rtimes_\pi G$ has the Haagerup property relative to $\LL(G)$.
\item $\LL(G)$ is quasi-regular inside $\Gamma(H_\R)\dpr \rtimes_\pi G$.
\end{enumerate}
\end{theo}

\begin{proof}
$(1) \Rightarrow (2)$. Since $\pi$ is assumed to be compact, write $\pi = \bigoplus_{j \in \N} \pi_j$ and $H_\R = \bigoplus_{j \in \N} H_\R^{(j)}$ with $\pi_j$ a finite dimensional orthogonal representation or $\pi_j = 0$. For $p \in \N$, let $E_p : H_\R \to \bigoplus_{0 \leq j \leq p} H_\R^{(j)}$ be the orthogonal projection and denote by $\Gamma(E_p) : \Gamma(H_\R)\dpr \rtimes_\pi G \to \Gamma(H_\R)\dpr \rtimes_\pi G$ the corresponding unique trace-preserving unital completely positive multiplier. Let ${\rm m}_{\rho_t} = E_M \circ \theta_t$ be the one-parameter family of trace-preserving unital completely positive maps which appeared in Section $\ref{preliminaries}$. Define ${\rm M}_{p, t} : \Gamma(H_\R)\dpr \rtimes_\pi G \to \Gamma(H_\R)\dpr \rtimes_\pi G$ by the formula ${\rm M}_{p, t} = {\rm m}_{\rho_t} \circ \Gamma(E_p)$. Then $({\rm M}_{p, t})_{p, t}$ is a family of $\LL(G)$-$\LL(G)$-bimodular trace-preserving unital completely positive maps which are compact over $\LL(G)$. Therefore $\Gamma(H_\R)\dpr \rtimes_\pi G$ has the Haagerup property relative to $\LL(G)$.

$(2) \Rightarrow (3)$. This follows from \cite[Proposition 3.4]{Po01}. 

$(3) \Rightarrow (1)$. Denote by $K_\R$ the unique closed $\pi(G)$-invariant subspace such that $\pi_K = \pi | K_\R$ is compact and $\pi_{H \ominus K} = \pi | H_\R \ominus K_\R$ is weakly mixing. By Corollary $\ref{weak-mixing-consequence}$, we get that 
$$\Gamma(H_\R)\dpr \rtimes_\pi G = \QN_{\Gamma(H_\R)\dpr \rtimes_\pi G}(\LL(G))\dpr \subset \Gamma(K_\R)\dpr \rtimes_{\pi_K} G.$$ 
Therefore $\pi = \pi_K$ and so $\pi$ is compact.
\end{proof}

\begin{cor}
Let $G$ be any countable discrete group and $\pi : G \to \mathcal O(H_\R)$ any compact orthogonal representation. Then $\Gamma(H_\R)\dpr \rtimes_\pi G$ has the Haagerup property if and only if $G$ has the Haagerup property.
\end{cor}

\begin{proof}
Assume that  $\Gamma(H_\R)\dpr \rtimes_\pi G$ has the Haagerup property. Then $\LL(G) \subset \Gamma(H_\R)\dpr \rtimes_\pi G$ has the Haagerup property and so does $G$ by \cite{choda}.

Assume that $G$ has the Haagerup property and $\pi : G \to \mathcal O(H_\R)$ is a compact orthogonal representation. Write $\pi = \bigoplus_{j \in \N} \pi_j$ and $H_\R = \bigoplus_{j \in \N} H_\R^{(j)}$ with $\pi_j$ a finite dimensional orthogonal representation or $\pi_j = 0$. For $p \in \N$, let $E_p : H_\R \to \bigoplus_{0 \leq j \leq p} H_\R^{(j)}$ be the orthogonal projection and denote by $\Gamma(E_p) : \Gamma(H_\R)\dpr \rtimes_\pi G \to \Gamma(H_\R)\dpr \rtimes_\pi G$ the corresponding unique trace-preserving unital completely positive multiplier.

Since $G$ has the Haagerup property, let $\varphi_n : G \to \C$ be a sequence of positive definite functions such that $\varphi_n(e) = 1$ for all $n$, $\lim_n \varphi_n = 1$ pointwise and $\varphi_n \in {\rm c}_0(G)$ for all $n \in \N$. Denote by $\Phi_n : \Gamma(H_\R)\dpr \rtimes_\pi G \to \Gamma(H_\R)\dpr \rtimes_\pi G$ the corresponding trace-preserving unital completely positive maps
$$\Phi_n(W(e_1 \otimes \cdots \otimes e_r) u_g) = \varphi_n(g) W(e_1 \otimes \cdots \otimes e_r) u_g.$$

Then we have that ${\rm M}_{n, p} = \Phi_n \circ \Gamma(E_p)$ forms a sequence of trace-preserving unital completely positive maps on $\Gamma(H_\R)\dpr \rtimes_\pi G$ such that $\lim_{n, p} \|{\rm M}_{n, p}(x) - x\|_2 = 0$ for all $x \in \Gamma(H_\R)\dpr \rtimes_\pi G$ and the corresponding bounded operators $T_{{\rm M}_{n, p}}$ are compact on $\LL^2(\Gamma(H_\R)\dpr \rtimes_\pi G)$ (see \cite[Lemma 3.3]{jolissaint}). Therefore $\Gamma(H_\R)\dpr \rtimes_\pi G$ has the Haagerup property.
\end{proof}

\bibliographystyle{plain}

\end{document}